\documentclass[notitlepage,11pt,reqno]{amsart}

\usepackage{amsopn,esint,nicefrac}
\usepackage[final]{hyperref}
\usepackage[T1]{fontenc}
\usepackage{graphicx}

\usepackage[utf8]{inputenc}%for appendix
\usepackage{apptools}%for appendix
\AtAppendix{\counterwithin{lemma}{section}} %for appendix
\usepackage[margin=1in]{geometry}
\allowdisplaybreaks
\newcommand{\stkout}[1]{\ifmmode\text{\sout{\ensuremath{#1}}}\else\sout{#1}\fi}
 
\newcommand{\R}{\mathbb{R}}

\newcommand{\Rn}{\mathbb{R}^n}
\usepackage{graphicx,enumitem,dsfont,upgreek}
 \usepackage[dvips]{epsfig}
 \usepackage[mathscr]{eucal}
\usepackage{amscd}
\usepackage{amssymb,nicefrac}
\usepackage{amsthm}
\usepackage{amsmath}
\usepackage{latexsym}
\usepackage{dsfont}
\usepackage{upref}
\usepackage{hyperref}

\usepackage{color}
\theoremstyle{plain}

\newtheorem{thm}{Theorem}[section]
\theoremstyle{plain}
\newtheorem{lem}[thm]{Lemma}
\newtheorem{prop}[thm]{Proposition}
\newtheorem{cor}[thm]{Corollary}

\newtheorem{defi}[thm]{Definition}%[section]
\newtheorem{rem}{Remark}[section]
\theoremstyle{definition}
\newtheorem*{maintheorem*}{Main Theorem}
\newtheorem*{maincorollary*}{Main Corollary}
{%
\setcounter{enumi}{0}

\begin{enumerate}}%
{\end{enumerate} }

{%
\setcounter{enumi}{0}

\begin{enumerate}}%
{\end{enumerate} }

\newcommand{\norm}[1]{\ensuremath{\left\|#1\right\|}}

\newcommand{\cone}{\ensuremath{\mathcal{C}}}
\newcommand{\cD}{\ensuremath{\mathcal{D}}}
\newcommand{\sL}{\ensuremath{\mathscr{L}}}

\newcommand{\sM}{\ensuremath{\mathscr{M}}}

\newcommand{\sV}{\ensuremath{\mathscr{V}}}
\newcommand{\trace}{\rm tr}
\newcommand{\grad}{\nabla}
\newcommand{\dist}{{\rm dist}}
\newcommand{\xbar}{\ensuremath{\bar{x}}}
\newcommand{\ybar}{\ensuremath{\bar{y}}}
\newcommand{\abar}{\ensuremath{\bar{a}}}

\newcommand{\snail}{{\rm snail}}
\newcommand{\av}{{\rm av}}
\newcommand{\data}{\texttt{data}}
\newcommand{\D}{{\rm d}}
\newcommand{\dx}{\ensuremath{\ {\rm d}x}}
\newcommand{\dy}{\ensuremath{\, {\rm d}y}}
\newcommand{\dz}{\ensuremath{\, {\rm d}z}}
\newcommand{\dt}{\ensuremath{\, {\rm d}t}}
\newcommand{\dr}{\ensuremath{\, {\rm d}r}}

\numberwithin{equation}{section} \allowdisplaybreaks

\usepackage{cancel,pdfsync}

\title[Gradient regularity of mixed local-nonlocal problems]{Interior $C^{1,\alpha}$ regularity of mixed local-nonlocal 
$(p,q)$-energy minimizers for $p\leq sq$}

\begin{document}

\author{Anup Biswas}

\address{Indian Institute of Science Education and Research-Pune, Dr.\ Homi Bhabha Road, Pashan, Pune 411008. INDIA Email:
{\tt anup@iiserpune.ac.in}}

\author{Erwin Topp}

\address{
Instituto de Matem\'aticas, Universidade Federal do Rio de Janeiro, Rio de Janeiro - RJ, 21941-909, BRAZIL; 
Email: {\tt etopp@im.ufrj.br}}

\begin{abstract}
We establish the local $C^{1, \alpha}$ regularity of minimizers for functionals of the form
$$w\mapsto \int_{\Omega}(|\grad w|^p-fw) \dx + \int_{\Rn}\int_{\Rn} \frac{|w(x)-w(y)|^q}{|x-y|^{n+sq}}\dx\dy,$$ 
where $s \in (0, 1)$, $1 < p \leq  sq$, and $f \in L^\infty(\Omega)$. This result complements the work of De Filippis and Minigione in \cite{DFM}, thereby completing the proof of $C^{1,\alpha}$ regularity for all $p, q \in (1, \infty)$ and $s \in (0, 1)$ with locally bounded source term.
\end{abstract}

\keywords{Gradient regularity, mixed local-nonlocal, H\"{o}lder regularity, double phase problems}
\subjclass[2020]{Primary: 35B65, 35J70, 35R09}

\maketitle

%\tableofcontents

\section{Introduction}
The study of the regularity of minimizers has been a central and active topic in modern analysis. Following the influential works of Caffarelli and Silvestre \cite{CS11, CS09}, there has been a significant surge of interest in nonlocal operators. These operators often exhibit features that pose new analytical challenges, which cannot be addressed using the classical techniques developed for local elliptic equations.

In recent years, mixed operators, formed by superimposing local and nonlocal components, have attracted considerable attention. The primary objective in this setting is to investigate variational problems that involve differential operators of distinct orders. A prototypical example of such a functional is
$$ w\mapsto \int_{\Omega} \left(\frac{1}{p}|\grad w|^p- f(x)w\right)\dx + \frac{1}{q}\iint_{\Rn\times\Rn} \frac{|w(x)-w(y)|^q}{|x-y|^{n+sq}}\dx\dy, \quad 0<s<1,$$
where $p, q\in (1, \infty)$ and $\Omega$ bounded. Note that for 
$p=q=2$, the corresponding Euler–Lagrange operator associated with the minimizer takes the form $-\Delta+(-\Delta)^s$.
 Some of the earliest studies on this operator were of a potential-theoretic nature and can be found in \cite{CKSV10, CKSV12, Foon}. More recently, a systematic investigation of this operator, including results on the maximum principle and regularity, was carried out by Biagi et al. in \cite{BDVV21, BDVV22}. For further developments and related results, we refer the readers to \cite{Bis, BMS23,DLV} and the references therein.
 
 In the nonlinear setting, the first work in this direction is due to Garain and Kinnunen~\cite{GK22}, where the case 
$p=q$ was considered. The authors established local boundedness of solutions, a Harnack inequality, and local H\"{o}lder continuity by employing the De Giorgi-Nash-Moser theory. Subsequent progress for 
$p=q$
was achieved by the first author and Lindgren~\cite{GL23}, where almost Lipschitz regularity of minimizers was obtained, along with local $C^{1, \alpha}$
 regularity whenever 
$sq>q-1$. A major advancement is recently made by De Filippis and Mingione \cite{DFM}, who proved local $C^{1, \alpha}$
 regularity under the weaker assumption 
$sq<p$. Under the same condition, they also established almost Lipschitz regularity up to the boundary, provided that 
$\Omega$ is a bounded $C^{1, \alpha_b}$
 domain. We also mention a few interesting works under the same set-up: \cite{AC25} establishing global $C^{1, \alpha}$ regularity,
 \cite{BKL24} for global Calder\'{o}n-Zygmund estimate in Reifenberg flat domains.

Our goal of this article is to investigate the case when $p\leq sq$. Let $\Omega$ be a bounded domain in $\Rn$ and we consider the functional
\begin{equation}\label{Func}
\mathcal{E}(w, \Omega)=\int_{\Omega} \left(\frac{1}{p}|\grad w|^p-fw\right)\dx + \frac{1}{2q}\iint_{C_\Omega} \frac{|w(x)-w(y)|^q}{|x-y|^{n+sq}}\dx\dy,
\end{equation}
where $C_\Omega := \R^N \times \R^N \setminus \Omega^c \times \Omega^c$ and $f\in L^\infty(\Omega)$. 
A function $u \in W^{1,p}(\Omega)$ is said to be a minimizer of $\mathcal{E}$ if $\mathcal{E}(u, \Omega)<\infty$ and
$$
\mathcal{E}(u, \Omega)\leq \mathcal{E}(w, \Omega),
$$
for all $w$ such that $\mathcal{E}(w, \Omega)<\infty$ and $u=w$ a.e. in $\Omega^c$. In this article, we assume that 
$$1<p\leq sq.$$

\subsection{Precise setting and result}
The existence of such a minimizer to \eqref{Func} can be   easily obtained if we set a regular enough boundary data. To see this,
we let $g\in W^{1, p}(\Omega)\cap W^{s, q}(\Omega_1)\cap L^{q}_{sq}(\Rn)$, where $\Omega\Subset\Omega_1$, and
$$
 L^{r}_{sq}(\Rn)=\{v\in L^{r}_{\rm loc}(\Rn)\; :\; \int_{\Rn} \frac{|v(y)|^{r}}{1+|y|^{n+sq}}\dy<\infty\}.
$$
It is then easily seen that $\mathcal{E}(g, \Omega)<\infty$. Furthermore, if $\Omega$ has a Lipschitz boundary, then using the lower semicontinuity property of the $W^{1, p}(\Omega)$-norm
and Rellich-Kondrachov compactness 

one can easily see that a minimizer exists in the class of functions $\mathbb{X}$, defined by,
$$\mathbb{X}(\Omega)=\{w\in (g+W^{1, p}_0(\Omega))\cap W^{s, q}(\Omega) :\; u=g\; \mbox{in}\; \Omega^c \}.$$
From the strict convexity of the functional, it immediately follows that the minimizer is unique. Additionally, if we let $g\in W^{s, q}(\Rn)$, then writing 
\begin{align*}
&\int_{\Omega}\left(\frac{1}{p}|\grad w|^p-fw\right) \dx + \frac{1}{2q}\int_{\Rn}\int_{\Rn} \frac{|w(x)-w(y)|^q}{|x-y|^{n+sq}}\dx\dy 
\\
&\qquad = \mathcal{E}(w, \Omega) + \frac{1}{2q}\iint_{\Omega^c\times\Omega^c} \frac{|g(x)-g(y)|^q}{|x-y|^{n+sq}}\dx\dy,
\end{align*}
for $w\in \mathbb{X}$, we can find a minimizer for the functional, mentioned on the lhs. From now on we assume $u$ is a minimizer of $\mathcal{E}$ as mentioned above. To define the weak formulation we need an appropriate class of test functions which will be essential to obtain the local H\"{o}lder regularity of the gradient.
Define 
\begin{align*}
\mathbb{X}_0(\Omega)=\{w\in W^{1, p}_0(\Omega)\cap W^{s, q}(\Rn)\; :\; w=0\; \mbox{in}\; \Omega^c\}.
\end{align*}

Since $u$ is a minimizer, for $w\in\mathbb{X}_0(\Omega)$ and $t\in (0, 1)$, we have $\mathcal{E}(u+tw, \Omega)-\mathcal{E}(u, \Omega)\geq 0$. Using convexity, this leads to
\begin{align*}
& t \int_\Omega |\grad u + t\grad w|^{p-2} (\grad u + t\grad w)\cdot \grad w \dx - t \int_{\Omega} f w \dx
\\
&\quad +\frac{t}{2} \iint_{C_\Omega} J_q(u(x)+tw(x)-u(y)-tw(y)) (w(x)-w(y)) \frac{\dx\dy}{|x-y|^{n+sq}}\leq 0,
\end{align*}
where $J_q(t)=|t|^{p-2} t$.
Now dividing by $t$, letting $t\to 0$ and applying 
the dominated convergence theorem, we arrive at
$$ \int_\Omega |\grad u|^{p-2} \grad u \cdot \grad w \dx - \int_{\Omega} f w \dx +\frac{1}{2} \iint_{C_\Omega} J_q(u(x)-u(y)) (w(x)-w(y)) \frac{\dx\dy}{|x-y|^{n+sq}}\leq 0.
$$

Since $-w\in \mathbb{X}_0(\Omega)$ and $w(x)-w(y)=0$ in $\Omega^c\times\Omega^c$, it is possible to obtain the reverse inequality in the last formula, from which we establish the weak formulation of the Euler-Lagrange equation of the minimizing problem as follows.
\begin{equation}\label{weaksol}
\int_\Omega (|\grad u|^{p-2}\grad u\cdot\grad w - fw)\dx + \frac{1}{2}\int_{\Rn}\int_{\Rn} J_q(u(y)-u(x)) (w(y)-w(x))\frac{\dx\dy}{|x-y|^{n+qs}}=0
\end{equation}
for all $w\in \mathbb{X}_0(\Omega)$. Another interesting observation is that for any $\tilde\Omega\Subset \Omega$, with
$\tilde\Omega$ having a Lipschitz boundary, we have $u$ as a minimizer of $w\mapsto \mathcal{E}(w, \tilde\Omega)$. To see this, we
consider $w$ such that $\mathcal{E}(w, \tilde\Omega)<\infty$ and $w=u$ in $\tilde\Omega^c$. Since $\partial\tilde\Omega$ in Lipschitz,
we have $w-u\in W^{1, p}_0(\tilde\Omega)$ and therefore, we can extend $w$ in $\Omega\setminus\tilde\Omega^c$ as $u$.
Again, since $\mathcal{E}(w, \tilde\Omega)<\infty$, one can easily verify that $\mathcal{E}(w, \Omega)<\infty$, implying
$\mathcal{E}(u, \Omega)\leq \mathcal{E}(w, \Omega)$. Again, since $(u(x), u(y))=(w(x), w(y))$ for 
$(x, y)\in C_{\Omega}\cap (\tilde\Omega^c\times \tilde\Omega^c)$, we obtain 
$\mathcal{E}(u, \tilde\Omega)\leq \mathcal{E}(w, \tilde\Omega)$.

Our starting point of this article is following result of \cite[Theorem~3]{DFZ}
\begin{thm}\label{Thm-cont}
Any $u\in W^{1,p}(\Omega)\cap W^{s, q}(\Omega)\cap L^{q}_{sp}(\Rn)$ which is a minimizer of $\mathcal{E}(\cdot, \Omega)$ is in $C^{0, \alpha}_{\rm loc}(\Omega)$ for some $\alpha>0$.
\end{thm}
This result lays the groundwork for transitioning to the viscosity solution framework, which plays a crucial role in improving the regularity of the solution to almost Lipschitz continuity. We then exploit this almost Lipschitz regularity to establish $C^{1,\alpha}$ regularity, thereby obtaining our main result
stated below.
\begin{thm}\label{T-main}
Any minimizer of \eqref{Func} is in $C^{1, \alpha}_{\rm loc}(\Omega)$ for some $\alpha\equiv \alpha(n, p, q, s)\in (0, 1)$. Furthermore, for any
$\Omega_0\Subset\Omega$ we have $[\grad u]_{C^{0, \alpha}(\Omega_0)}\leq c$,  where the constant $c$ depends on $\data$ and $\dist(\Omega_0, \partial\Omega)$.
\end{thm}
Here $\data$ is used as a shorthand notation to denote the following set of parameters.
$$\data:=(n, p, q, s, \norm{u}_{L^\infty(\Omega)}, \norm{u}_{L^{q-1}_{sq}(\Rn)}, \norm{f}_{L^\infty(\Omega)}).$$
At this point, we mention a few relevant works and discuss the limitations of the existing tools in addressing Theorem~\ref{T-main}.
A close resemblance of our model can be found with double-phase problems of $(p, q)$-Laplacian type, for which the $C^{1,\alpha}$ regularity has been extensively studied (see, for instance, \cite{BCM18, CM15, CM15a, FM00}). We also refer to \cite{BLS25}, where the authors investigate a double-phase problem involving both local and nonlocal operators, with the modulating coefficient influencing the local operator. 
In the context of the fractional $q$-Laplacian, it is known that weak solutions are $\min\{\frac{sq}{q-1}, 1\}$-H\"{o}lder continuous, while fractional $q$-harmonic functions are $\min\{\frac{sq}{q-2}, 1\}$-H\"{o}lder continuous for $q>2$ and Lipschitz continuous for $q \in (1,2]$; see, for instance, \cite{BS25,BT25,BDLMS24a,BDLMS24b,BLS18,GL24}. These regularity estimates are, in general, sharp.
A recent breakthrough due to Giovagnoli, Jesus, and Silvestre \cite{GJS} establish $C^{1,\alpha}$ regularity for fractional $q$-harmonic functions when $q \in [2, \frac{2}{1-s})$. It was commonly believed that, for $p \leq sq$, the fractional $q$-Laplacian would be the dominant term in \eqref{weaksol}. Consequently, based on the above discussion, one would not expect more than H\"{o}lder continuity for the minimizer $u$.
From this viewpoint, Theorem~\ref{T-main} is somewhat counterintuitive, if not genuinely surprising.

To comment on the proof of Theorem~\ref{T-main}, we note that when $sq < p$, the fractional term behaves as a lower-order perturbation. This is due to the fact that the $W^{s,q}$ norm can be controlled by the $W^{1,p}$ norm 
(see \cite[Section~2.2]{DFM} and the localization argument in \cite[Section~4]{DFM}). This observation plays a crucial role in improving the regularity of the minimizer to almost Lipschitz continuity, and in constructing a suitable test function for the excess-decay argument leading to the $C^{1,\alpha}$ estimate (see \cite[Lemma~6.2]{DFM}). Since, in the present setting, we do not have this advantage, we adopt a slightly unconventional approach (at least for the class of problems that we consider). Using Theorem~\ref{Thm-cont}, we transition to the viscosity framework, where we employ a nonlocal version of the 
Jensen–Ishii lemma in an iterative manner to bootstrap the regularity up to almost Lipschitz continuity.
We also mention 
\cite{BCI11,BCCI12,BS25,BT25,BT24}, where nonlocal
Ishii-Lions argument have been employed in establishing H\"{o}lder/Lipschitz regularity results.
We then return to the weak formulation to derive the excess-decay estimate.
In order to construct a suitable test function, we are required to establish that the solution of
$$-\Delta_p v=0\quad \text{in}\; \Omega, \quad \text{and}\quad v=g\quad \text{on}\; \partial\Omega,$$
belongs to $C^{0,\beta}(\overline{\Omega})$, whenever $g \in C^{0,\beta}(\overline{\Omega})$ for some $\beta \in (0,1)$ and $\Omega$ is a $C^2$ domain.
To the best of our knowledge, most existing results on the boundary regularity of $p$-harmonic functions \cite{CD89,DGK04,Giu03} assume boundary data $g$ belonging to a suitable Sobolev space, which is strictly smaller than $C^{0,\beta}(\overline{\Omega})$. Although such a result was long expected for H\"{o}lder continuous boundary data, it had not been established previously, as far as we are aware. We resolve this issue here by employing the viscosity solution approach.  Once the appropriate test function is constructed, the excess-decay estimate follows by a standard argument as outlined in \cite{DFM}, leading to the local $C^{1,\alpha}$ regularity estimate. We also remark that our proof of Theorem~\ref{T-main} continue to hold for continuous weak solutions without any modifications.

We conclude this section with the following remark highlighting possible extension and limitation of our technique.
\begin{rem}
We make the following observations. 
\begin{itemize}
\item[(i)] Our proof extends verbatim to symmetric nonlocal kernels $K$ that are comparable to the 
$(q, s)$-fractional kernel; specifically,
$$\frac{\lambda}{|y|^{n+sq}}\leq K(y)\leq \frac{\Lambda}{|y|^{n+sq}}, \quad 0<\lambda\leq \Lambda.$$
We also expect that the local energy density $|\nabla u|^{p}$ can be replaced by a more general functional $F(\nabla u)$ satisfying the usual structural conditions compatible with the viscosity framework. However, some of the estimates developed in this paper—for instance, the bound on $\mathcal{A}_\alpha$ in Lemma~\ref{L2.2}—appear difficult to generalize to such broader settings. The main obstruction stems from the variant of the nonlocal Jensen–Ishii lemma adopted from \cite{BI08}. In contrast to the classical version in \cite[Theorem~3.2]{CIL}, this variant does not provide sufficiently strong control on the norms of the coupling matrices $X_\alpha$ and $Y_\alpha$ (see \eqref{E2.8}), making the extension to general models challenging.

\item[(ii)] We also do not consider $f\in L^p(\Omega)$ for $p<\infty$. Since a major part of our proofs relies on the theory of viscosity solutions, it is convenient for us to assume that the source term $f$ is bounded. We believe that, by employing a perturbation-type argument (cf. \cite{BLS18}), the H\"{o}lder regularity of the gradient can be extended to a suitable class of integrable functions.
\end{itemize}
\end{rem}
The remainder of the article is organized as follows. In Section~\ref{S-visco}, we introduce the viscosity framework along with the nonlocal Jensen–Ishii lemma. Section~\ref{S-lip} is devoted to establishing the almost Lipschitz regularity, while Section~\ref{S-bdry} addresses the boundary regularity for $p$-harmonic functions. Finally, in Section~\ref{S-alpha}, we prove the $C^{1,\alpha}$ regularity result.

Throughout the paper, $\kappa, \kappa_1, \kappa_2, \ldots$ denote generic constants that may vary from line to line.

%%%%%%%%%%%%%%%%%%%%%%%%%%%%
\section{Viscosity setting and preliminaries}\label{S-visco}
In this section, we introduce several tools from the theory of viscosity solutions for integro-differential equations. The starting point is Theorem~\ref{Thm-cont}, which ensures the continuity of the minimizer.
There are numerous works in the literature that establish the equivalence between weak and viscosity solutions; see, for instance, \cite{JLM01} for the $q$-Laplacian and \cite{BM21, KKL19} for the fractional $q$-Laplacian. We use this last approach here, which we briefly explain next for completeness.

Let us introduce the following notation:
\begin{align*}
& \sL u = - \Delta_p u + \sL_q u, \ \mbox{where}
\\
& \Delta_p u = {\rm div}(|\grad u|^{p-2}\grad u) , \quad \sL_q u = {\rm PV}\int_{\Rn} J_q(u(x)-u(x+z)) \frac{\dz}{|z|^{n+sq}},
\end{align*}
and, where ${\rm PV}$ stands for the Cauchy Principal Value, that is,
$$
{\rm PV}\int_{\Rn} J_q(u(x)-u(x+z)) \frac{\dz}{|z|^{n+sq}} = \lim_{\epsilon \to 0^+} \int_{\Rn \setminus B_\epsilon} J_q(u(x)-u(x+z)) \frac{\dz}{|z|^{n+sq}}.
$$

With these definitions, we consider the equation
\begin{equation}\label{eqEL}
\sL u = f \quad \mbox{in $\Omega$},
\end{equation}
and we say that $u \in W^{1, p}_{\rm loc}(\Omega)\cap W^{s, q}_{\rm loc}(\Omega)\cap L^{q-1}_{sp}(\Rn)$ is {\bf a local weak solution} to~\eqref{eqEL} if it satisfies~\eqref{weaksol} for all $w \in \mathbb X_0(\Omega)$. Thus, as we saw in the Introduction, minimizers of $\mathcal E$ are weak solutions to~\eqref{eqEL}.

\smallskip

For simplicity, we denote by $F$ the $p$-Laplacian in its \textit{non-variational form}, namely $\Delta_p u(x) = F(Du(x), D^2u(x))$ with
\begin{equation}\label{defF}
F(\xi, X) := |\xi|^{p-2}{\trace} X + (p-2) |\xi|^{p-4} \langle \xi X, \xi\rangle\quad \text{for}\; \xi \in \R^n, X \in \mathbb S^{n}.
\end{equation}

In some cases depending on the parameters $p$ and $q$, local and nonlocal operators are sensitive to pointwise evaluation at critical points of the function $u$.
The map $x \mapsto \sL_q u(x)$ is known to be classically defined and continuous at $x \in B_r(x)$ for $u\in C^2(B_r(x))\cap L^{q-1}_{sq}(\Rn)$ for some $r>0$ if $\grad u(x)\neq 0$ or $q > \frac{2}{2-s}$, see \cite{KKL19}. Next class of functions is used to treat some complementary case. Given an open set $D$ and $\beta > 0$, we denote by $C^2_\beta(D)$, a subset of $C^2(D)$, defined as
$$
C^2_\beta(D)=\left\{\phi\in C^2(D)\; :\; \sup_{x\in D}\left[\frac{\min\{d_\phi(x), 1\}^{\beta-1}}{|\nabla\phi(x)|} +
\frac{|D^2\phi(x)|}{(d_\phi(x))^{\beta-2}}\right]<\infty\right\},
$$
where
$$ 
d_\phi(x)=\dist(x, N_\phi)\quad \text{and}\quad N_\phi=\{x\in D\; :\; \nabla\phi(x)=0\}.$$

The above restricted class of test functions appears to be necessary to establish a connection with the viscosity theory, since it allows us to define $\sL_q$ when $q \leq 2/(2-s)$. Similarly, this is also a matter of fact for the local part of $\sL$. In fact, $F$ in~\eqref{defF} is singular at $\xi=0$ if $p < 2$. For viscosity evaluation, given $(\xi, X) \in \R^n \times \mathbb S^n$, we define the lower semicontinuous relaxation of $F$ as
$$
F_*(\xi, X) = \lim \limits_{\epsilon \searrow 0} \inf \{ F(\xi', X') : 0 < |(\xi', X') - (\xi, X)| < \epsilon \},
$$
and in the same way, we define the upper semicontinuous relaxation of $F$ as $F^* = -(-F)_*$. Notice that if $p \geq 2$, then $F$ is continuous in all its arguments and $F^* = F_* = F$.

If $p < 2$ and $\beta \geq \frac{p}{p-1}$, notice that for $\phi \in C_\beta^2$ and $x_0$ an isolated critical point of $\phi$, we readily have that
$$
|D\phi(x)|^{p-2} |D^2 \phi(x)| 
$$
remains bounded as $x \to x_0$. This implies that $F_*, F^*$ are well-defined at for such as test functions $\phi$ at critical points.

Now we are ready to define the viscosity solution, which is basically a combination of~\cite[Definition 2.1]{CGG91} and \cite[Definition~3]{KKL19}.

\begin{defi}\label{Def1.1}
A measurable function $u:\Rn\to \R$, upper (resp. lower) semi continuous in $\Omega$ with $u^+ \in L^{q-1}_{sq}(\R^n)$ (resp. $u^- \in L^{q-1}_{sq}(\R^n)$) is a viscosity subsolution (resp. supersolution) to~\eqref{eqEL} in $\Omega$ if for each $x_0 \in \Omega, r > 0$ with $B_r(x_0)\subset \Omega$, and each $\phi\in C^2(B_r(x_0))$ such that $\phi(x_0)=u(x_0)$,
$\phi\geq u$ in $B_r(x_0)$ (resp. $\phi \leq u$ in $B_r(x_0)$), satisfying one of the following conditions
\begin{itemize}
\item[(a)] $p \geq 2$ and $q>\frac{2}{2-s}$, or $\nabla\phi(x_0)\neq 0$,

\item[(b)] $\nabla\phi(x_0)= 0$, $x_0$ is an isolated critical point, and
$\phi\in C^2_\beta(B_r(x_0))$ for some $\beta \geq \frac{p}{p-1}$ if $1 < p < 2$, and $\beta > \frac{sq}{q-1}$ if $q \leq \frac{2}{2-s}$,
\end{itemize}

then we have 
\begin{align*}
& F_*(\nabla \phi(x_0), D^2 \phi(x_0)) + \sL_q \phi_r(x_0)\leq f(x_0) \\
& (resp. \quad F^*(\nabla \phi(x_0), D^2 \phi(x_0)) + \sL_q \phi_r(x_0)\geq f(x_0))
\end{align*}
where
\[
\phi_r(x)=\left\{\begin{array}{ll}
\phi(x) & \text{for}\; x\in B_r(x_0),
\\[2mm]
u(x) & \text{otherwise}.
\end{array}
\right.
\]

We say $u$ is a viscosity solution to $\sL u = f$ in $\Omega$, if it is both sub and super solution in $\Omega$. 
\end{defi}

This (admittedly confusing) notion of solution seeks for a slightly larger class of test functions at every point, including the test functions with vanishing gradient. It is adequate for dealing with the existence issues by approximation, for example, through the natural ``vanishing viscosity method" with 
$$
F_\mu = \mathrm{div} \Big{(} (|\nabla u|^2 + \mu^2)^{\frac{p-2}{2}} \nabla u \Big{)} %(\delta + |\xi|)^{p - 2} \Big{(} \mathrm{Tr}(X) + (p - 2)\langle \hat \xi, X \hat \xi \rangle \Big{)}, 
$$
in place of $F$ and send $\mu \searrow 0$.

For the purposes of this article, we use in an extensive way only the case when test function has non-vanishing gradient at test points, specially in the proof of Theorem~\ref{T-main1} below, though some properties can be handled for more general cases. This is the aim of the following

% note that any $u\in W^{1, p}_{\rm loc}(\Omega)\cap W^{s, q}_{\rm loc}(\Omega)\cap C(\Omega)\cap L^{q-1}_{sp}(\Rn)$
% that satisfies \eqref{weaksol} is also a viscosity solution at the non-critical points in $\Omega$. \footnote{ET: I agree that it is enough to deal with the case $\nabla \varphi(x_0) \neq 0$. In fact, in~\cite{JLM} they do not consider vanishing gradient test functions in the singular case. Not sure how to present the definition, but something must be said, in my opinion. {\cb AB: I have fixed the definition. We have to take $\beta>\max\{\frac{p}{p-1}\frac{sq}{q-1}\}$, then $F^*$ and
% $F_*$ make sense. In fact, we can prove Prop 2.2 at critical point with suitable modification}}
%In view of Theorem~\ref{Thm-cont} and the above discussion, a minimizer thus turns out to be a viscosity solution at all non-critical points.

\begin{prop}\label{P1.2}
Let $u \in W^{1, p}_{\rm loc}(\Omega)\cap W^{s, q}_{\rm loc}(\Omega)\cap C(\Omega)\cap L^{q-1}_{sp}(\Rn)$ be a weak solution to~\eqref{eqEL}, as defined in \eqref{weaksol}.
Let $x\in\Omega, r > 0$ such that $B_r(x)\Subset \Omega$, and assume there exists $\varphi\in C^2(B_r(x))$, $\varphi\geq u$ in $B_r(x)$ with $\varphi(x)=u(x)$ such that case $(a)$ in Definition~\ref{Def1.1} holds. Then, $\sL\varphi_r(x)$ exists and satisfies
$$
\sL \varphi_r (x)\leq \norm{f}_\infty,
$$
where
\[
\varphi_r=\left\{\begin{array}{ll}
\varphi & \text{in}\; B_r(x),
\\[2mm]
u & \text{otherwise}.
\end{array}
\right.
\]

 Similarly, if $\varphi\in C^2(B_r(x))$, $\varphi\leq u$ in $B_r(x)$ such that case $(a)$ in Definition~\ref{Def1.1} holds, then $\sL\varphi_r(x)\geq -\norm{f}_\infty$.
\end{prop}

\begin{proof}
We only prove the first part, and the proof for the second part would be analogous. Let $x\in\Omega$ and $\varphi\in C^2(B_r(x))$ be
a test function touching $u$ from above at $x$ and $\grad \varphi(x)\neq 0$. Since
$\varphi_r(y)-\varphi_r(x)\geq \varphi_\delta(y)-\varphi_\delta(x)$ for any $\delta\leq r$ and $y\in B_r(x)$, from the monotonicity of $J_q$, it is enough to show that
$\sL\varphi_\delta(x)\leq  \norm{f}_\infty$ for some $\delta\leq r$. First, we choose $\delta$ small enough so that $\varphi\in C^2(\overline{B_{2\delta}(x)})$ and $|\grad\varphi|>0$ in 
$\overline{B_{2\delta}(x)}$.
Suppose that
$\sL\varphi_{\delta}(x)> \norm{f}_\infty+\eta$ for some $\eta>0$. 
From \cite[Lemma~3.6]{KKL19} we recall that the nonlocal integral is classically defined in $B_{2\delta}(x)$. In fact, using the continuity of 
$y\mapsto \sL\varphi_{\delta}(y)$ in $\overline{B_\delta(x)}$ (see \cite[Lemma~3.8]{KKL19}), we can find $\delta_1\leq \delta$ such that
$$
\sL\varphi_{\delta}(y)\geq \norm{f}_\infty+\eta/2\quad \text{in}\; \overline{B_{\delta_1}(x)}.
$$

Now consider a smooth, non-negative cutoff function $\chi$, supported in $B_{\delta_1}(x)$ and $\chi(x)=1$. Using the argument in \cite[Lemma~3.9]{KKL19} (see (3.6) there), we can find a 
$\theta\in (0, 1)$ small enough so that 
$$\sup_{y\in B_{\delta_1}(x)}|\sL \varphi_{\delta}(y)-\sL\tilde\varphi_{\delta}(y)|<\eta/4,$$
where $\tilde\varphi_\delta=\varphi_\delta-\theta\chi$. This, in turn, gives us 
\begin{equation}\label{EP1.2A}
\sL\tilde\varphi_{\delta}(y)\geq \norm{f}_\infty+\eta/4\quad \text{in}\; B_{\delta_1}(x).
\end{equation}
Denote by $D=B_{\delta_1}(x)$. We claim that for any $v\in  \mathbb{X}_0(D)\subset  \mathbb{X}_0(\Omega)$, $v\geq 0$, we have
\begin{equation}\label{EP1.2B}
\int_D |\grad{\tilde\varphi_\delta}|^{p-2}\grad{\tilde\varphi_\delta}\cdot\grad{v}\dz
+ \frac{1}{2}\int_{\Rn}\int_{\Rn} J_q(\tilde\varphi_\delta(z)-\tilde\varphi_\delta(y))(v(z)-v(y))\frac{\dz\dy}{|z-y|^{n+qs}}\geq (\norm{f}_\infty+\eta/4)\int_{D} v \dz.
\end{equation}

Multiply \eqref{EP1.2A} by $v$ and integrate both sides over $D$. Using integration-by-parts we can easily see the first term in \eqref{EP1.2B} coming from the $p$-Laplacian. Thus, it is enough to
prove that 
\begin{equation}\label{EP1.2C}
\int_D v(z) \sL_q \tilde{\varphi}_\delta(z) \dz= \frac{1}{2}\int_{\Rn}\int_{\Rn} J_q(\tilde\varphi_\delta(z)-\tilde\varphi_\delta(y))(v(z)-v(y))\frac{\dz\dy}{|z-y|^{n+qs}}.
\end{equation}
From \cite[Lemma~3.6]{KKL19} we see that given any $\varepsilon>0$, there exists $\kappa\in (0, \delta_1)$ such that
$$ \Big| {\rm PV}\int_{B_\kappa(z)} J_q(\tilde\varphi_\delta(z)-\tilde\varphi_\delta(y)) \frac{\dy}{|z-y|^{n+sq}} \Big|\leq \varepsilon$$
for all $z\in D$. Again, another use of integrating by parts  gives us
\begin{align*}
& \int_D v(z) \int_{|z-y|\geq \kappa} J_q(\tilde\varphi_\delta(z)-\tilde\varphi_\delta(y)) \frac{\dy\dx}{|z-y|^{n+qs}} \dz
\\
&\quad =\int_{\Rn} v(z) \int_{|z-y|\geq \kappa} J_q(\tilde\varphi_\delta(z)-\tilde\varphi_\delta(y)) \frac{\dy}{|z-y|^{n+qs}} \dz
\\
&\quad =\frac{1}{2}\int_{\Rn}\int_{\Rn} 1_{\{|z-y|\geq \kappa\}} J_q(\tilde\varphi_\delta(z)-\tilde\varphi_\delta(y))(v(z)-v(y)) \frac{\dy\dz}{|z-y|^{n+qs}} .
\end{align*}
Since $\tilde\varphi_\delta\in W^{s, q}(B_{2\delta}(x))\cap L^{q-1}_{sq}(\Rn)$, using the dominated convergence theorem, we
can let $\kappa\to 0$ and from the arbitrariness of $\varepsilon$ we have \eqref{EP1.2C}. This proves our claim \eqref{EP1.2B}. 

Using  \eqref{weaksol} and \eqref{EP1.2B}, we next prove that $u\leq \tilde\varphi_\delta$ in $\Rn$, from which we arrive at a contradiction by the construction of $\tilde \varphi_\delta$. 
In fact, take $v=(u-\tilde\varphi_\delta)_+$. Since $\chi=0$ on  $B^c_{\delta_1}(x)$, we have $v\in W^{1,p}_0(D)$. Also, since $v\in W^{s, q}(B_{2\delta}(x))$
and $v=0$ in $B_{2\delta}(x)\setminus D$, we have $v\in W^{s,q}(\Rn)$  \cite[Lemma~5.1]{DNPV}, and therefore, $v\in \mathbb{X}_0(D)\subset \mathbb{X}_0(\Omega)$. Thus $v$ is a valid
test function for  \eqref{weaksol} and \eqref{EP1.2B}. Subtracting the relevant in-equations  we arrive at
\begin{align*}
& I_1 + I_2 \leq -\frac{\eta}{4} \int_D v \dy, \quad \mbox{where} 
\\
& I_1 := \int_D (|\grad u|^{p-2}\grad u- |\grad \tilde\varphi_\delta|^{p-2}\grad \tilde\varphi_\delta)\cdot \grad v \dy, \\
& I_2 : =  \frac{1}{2}\int_{\Rn}\int_{\Rn} \left(J_q(u(z)-u(y))-
J_q(\tilde\varphi_\delta(z)-\tilde\varphi_\delta(y))\right)(v(z)-v(y))\frac{\dz\dy}{|z-y|^{n+qs}}.
\end{align*}

For the first term in the rhs above we have
\begin{align*}
I_1 = & (p - 1) \int_{D} \int_{0}^{1} |\nabla (\tilde \varphi_\delta + t(u - \tilde \varphi_\delta))|^{p-2}  \dt\; \nabla (u - \tilde \varphi_\delta) \cdot \nabla (u - \tilde \varphi_\delta)_+ \dy 
\\
= & (p - 1) \int_{\{ u > \tilde \varphi_\delta\}} \int_{0}^{1} |\nabla (\tilde \varphi_\delta + t(u - \tilde \varphi_\delta))|^{p-2}  \dt\; |\nabla (u - \tilde \varphi_\delta)_+|^2 \dy,
\end{align*}
from which we conclude that $I_1 \geq 0$. 
Similarly, for $I_2$, denoting $\Delta_{z,y} f = f(z) - f(y)$, we see that
\begin{align*}
J_q(\Delta_{z,y} u) - J_q(\Delta_{z,y} \tilde \varphi_\delta) = (q - 1) \int_{0}^{1} |\Delta_{z, y} u + t(\Delta_{z, y}(\tilde \varphi_\delta - u))|^{q-2} \dt\;  (\Delta_{z,y}(\tilde \varphi_\delta - u)),
\end{align*}
from which, using that $v = (u - \tilde \varphi_\delta)_+$ and $ (a-b)(a_+-b_+)\geq (a_+-b_+)^2$ for all $a, b\in \R$, we get
\begin{align*}
I_2 \geq \frac{q - 1}{2}\int_{\Rn}\int_{\Rn} \int_{0}^{1}  |\Delta_{z, y} \tilde \varphi_\delta + t(\Delta_{z, y}(u - \tilde \varphi_\delta))|^{q-2} \dt\; (\Delta_{z, y}(u - \tilde \varphi_\delta)_+)^2\frac{\dz\dy}{|z-y|^{n+qs}},
\end{align*}
which is a nonnegative quantity too. Hence, we conclude that 
$$0 \leq -\int_D v\, \dy \Rightarrow \int_D v \dy= 0,$$
 and therefore $v \equiv 0$ in $D$ by continuity. Thus, $u \leq \tilde \varphi_\delta$ in $D$, which gives $u(x)\leq \tilde \varphi_\delta(x)=\varphi_\delta(x)-\theta=u(x)-\theta$, leading to a contradiction.
%We note that $v(z)>v(y)$ implies $v(z)=u(z)-\tilde\varphi_\delta(z)>0$ and 
%$u(z)-\tilde\varphi_\delta(z)\geq u(y)-\tilde\varphi_\delta(y)\Rightarrow u(z)-u(y)>\tilde\varphi_\delta(z)-\tilde\varphi_\delta(y)$.
%Since $J_q$ is increasing, it is now easy to see that  the second integration above is non-negative. Again, since $t\mapsto |t|^{p-2}t$ is increasing,
%and $\grad v =1_{\{v>0\}} (\grad u-\grad \tilde\varphi_\delta)$, the first integration on the left is also nonnegative. Hence 
%$\int_D v \dy=0$ implying $v=0$ in $D$. This, in turn, gives $u\leq \tilde\varphi_\delta$ in $\Rn$. But 
%$u(x)> \varphi_\delta(x)-\theta=\tilde\varphi_\delta(x)$, which is a contradiction. Thus we must have 
%$\sL\varphi_{\delta}(x)\leq  \norm{f}_\infty$ for $\delta$ as chosen above. 
This completes the proof.
\end{proof}

\begin{rem}
It is possible to extend the last proposition to case $(b)$ in Definition~\ref{Def1.1} with $p \geq 2$ and $\beta > \frac{qs}{q-1}$ if $1 < q \leq \frac{2}{2 - s}$. As can be seen in~\cite[Lemma 3.8]{KKL19}, $x \mapsto \sL \phi_\delta(x)$ is continuous around test points $x$ in this case, which is the useful property to reproduce the perturbation argument in the proof. This seems to be more difficult to adapt in case $1 < p < 2$ by the degeneracy of the operator $F$ and the necessity to use its relaxed version. 
\end{rem}

Since our operator is a superposition of operators of local and non-local type, we rely on the theory viscosity solution developed by Barles and Imbert in \cite{BI08}.
More precisely, our regularity estimate uses the nonlocal Jensen-Ishii lemma of \cite{BI08}. To introduce it, we need the notion of subjets and superjets.
By $\mathbb{S}^n$ we denote the set of all real $n\times n$ symmetric matrices.
Given  $x\in\Omega$, we define the superjet as
$$
J^+u(x)=\{(\xi, X)\in \Rn\times\mathbb{S}^n\; :\; u(x+h)\leq u(x) + \xi\cdot h + \frac{1}{2}\langle h X, h\rangle + o(|h|^2)\},
$$
and its limit set
\begin{align*}
\bar{J}^+ u(x) = \{(\xi, X)\in \Rn\times\mathbb{S}^n\; &:\; \exists (x_m, \xi_m, X_m)\in \Omega\times\Rn\times\mathbb{S}^n\; \text{such that}\; (\xi_m, X_m)\in J^+u(x_m)
\\
&\quad \text{and} \; (x_m, u(x_m), \xi_m, X_m)\to (x, u(x), \xi, X)\}.
\end{align*}

Subjets $J^{-}u(x)$ and its limit set $\bar{J}^{-}u(x)$ are defined in an analogous fashion. 

Using Proposition~\ref{P1.2}, monotonicity of $J_q$ and~\cite[Proposition~1]{BI08} we then obtain the following
\begin{lem}\label{P1.3}
Let $u \in W^{1, p}_{\rm loc}(\Omega)\cap W^{s, q}_{\rm loc}(\Omega)\cap C(\Omega)\cap L^{q-1}_{sp}(\Rn)$ be a weak subsolution to~\eqref{eqEL}.
Let $x\in\Omega, r > 0$ with $B_r(x)\Subset \Omega$ such that there exists $\varphi\in C^2(B_r(x))$, $\varphi\geq u$ in $B_r(x)$ with $\varphi(x)=u(x)$ and case $(a)$ in Definition~\ref{Def1.1} holds. If $(\xi, X)\in J^+u(x)$ with $\xi=\grad\varphi(x)$ and $X\leq D^2\varphi(x)$, then we have
$$
 - F(\xi, X) + \sL_q \varphi_r(x)\leq \norm{f}_\infty,
$$
where $\varphi_r$ is given by Proposition~\ref{P1.2}.

An analogous conclusion holds for supersolutions.
\end{lem}
It is helpful to note that by the nature of our operator, it is not necessary for $\varphi$ to touch $u$ at $x$ (be it from above or below). It is enough if $\varphi-u$ attains its minimum (or maximum) at $x$ in $B_r(x)$,
 since we can always translate $\varphi$ to meet this criterion.

Now given a function $\phi$, we define the sup-convolution, for $\alpha>0$, as
$$
R^{\alpha}[\phi](z, \xi)=\sup_{|Z-z|\leq 1}\left\{ \phi(Z)-\xi\cdot (Z-z)-\frac{|Z-z|^2}{2\alpha}\right\}.
$$
It is known from \cite[Proposition~3]{BI08} that if $\phi\in C^2(B)$ for some ball $B$, then for any $B_1\Subset B$, there exists $\alpha_0$ small such that
$R^\alpha[\phi]\in C^2(B_1)$ for all $\alpha\leq \alpha_0$. Furthermore, $R^\alpha[\phi](\cdot, \grad\phi)\to \phi$ in $C^2(\bar{B}_1)$ as $\alpha\to 0$.
We recall the following nonlocal Jensen-Ishii's lemma from \cite[Lemma~1]{BI08}, see also Remark 4.5 in~\cite{CLLT}.

\begin{lem}\label{LJI}
Let $u$ and $v$ be usc and lsc, respectively, in $\Rn$. Suppose $(\xbar, \ybar)$ be a global maximum of the function $u(x)-v(y)-\phi(x, y)$ in $\Rn\times\Rn$ with $\xbar,\ybar\in \Rn$, $\phi\in C^2(B_\delta(\xbar,\ybar))$ and
$\bar\xi_x=\grad_x\phi(\xbar,\ybar), \bar\xi_y=\grad_y\phi(\xbar, \ybar)$. Then the following hold: for every $\delta_1<\delta$ there exists $\alpha_1=\alpha(\delta_1)$
such that for all $\alpha\leq \alpha_1$, there are points $x_k\to \xbar,
y_k\to \ybar, p_{k}\to \bar \xi_x, q_{k}\to \bar \xi_y$ and matrices $X_k, Y_k\in \mathbb{S}^n$, and a sequence of function $\phi_k$ satisfying
\begin{enumerate}
\item $(x_k, y_k)$ is a global maximum of $u-v-\phi_k$.
\item $u(x_k)\to u(\xbar)$ and $v(y_k)\to v(\ybar)$. $(p_k, X_k)\in J^{+} u(x_k)$ and $(-q_k, Y_k)\in J^{-}u(y_k)$.
\item $\phi_k\to \phi_\alpha:=R^\alpha[\phi](\cdot, (\bar\xi_x, \bar\xi_y))$ in $C^2(B_{\delta_1}(\xbar, \ybar))$.
\item 
$$
-\frac{1}{\alpha} I\leq 
\begin{pmatrix}
 X_k & 0\\
 0 & -Y_k
\end{pmatrix}
\leq D^2\phi_k(x_k, y_k).
$$
\end{enumerate}
Moreover, $p_k=\grad_x \phi_k(x_k, y_k), q_k= \grad_y \phi_k(x_k, y_k), \phi_\alpha(\xbar,\ybar)=\phi(\xbar,\ybar)$ and
$\grad\phi_\alpha(\xbar,\ybar)=\grad\phi(\xbar,\ybar)$.
\end{lem}

We need a few technical lemmas. For $D \subseteq \Rn$ measurable, we introduce the notation
$$
\sL_q[D]u(x):={\rm PV} \int_D J_{q}(u(x)-u(y))\frac{\dy}{|x-y|^{n+qs}}.
$$

\begin{lem}\label{lemcontnonlocal}
For  $u\in L^{q-1}_{sq}(\Rn)$ and a sequence of points $x_k\to \xbar$, as $k\to\infty$, if we have $u(x_k)\to u(\xbar)$, then,
for $\delta_1>0$,
$$
\int_{|z|\geq \delta_1} J_q(u(x_k)-u(x_k+z))\frac{\dz}{|z|^{n+qs}}\longrightarrow \int_{|z|\geq \delta_1} J_q(u(\xbar)-u(\xbar+z))\frac{\dz}{|z|^{n+qs}}.
$$
\end{lem}

\begin{proof}
Since $u\in L^{q-1}_{sq}(\Rn)$, we can find a sequence of $\chi_m\in C_c(\Rn)$ such that
$$ \int_{\Rn}\frac{|u-\chi_m|^{q-1}}{1+|z|^{n+sq}} \dz\to 0\quad \text{as}\; m\to\infty.$$
Let $\mathcal{K}=\{\xbar\}\cup\{x_1, x_2, \ldots\}$. Form our assertion, we have $|u(y)|\leq \kappa$ for $y\in\mathcal{K}$ and for some
constant $\kappa$. 
Since, for $a, b\in\R$, we have
\[|J_q(a)-J_q(b)|\leq \left\{\begin{array}{ll}
2^{q-2}(q-1) (|a|+|b|)^{q-2}|a-b| &\text{for}\; q>2,
\\[2mm]
2|a-b|^{q-1}& \text{for}\; q\in (1,2],
\end{array}
\right.
\]
 it follows that, for $q>2$ and $y\in\mathcal{K}$,
\begin{align*}
&\left|\int_{|z|\geq \delta_1} J_q(u(y)-u(y+z))\frac{\dz}{|z|^{n+sq}}-\int_{|z|\geq \delta_1} J_q(u(y)-\chi_m(y+z))\frac{\dz}{|z|^{n+sq}}\right|
\\
&\leq 2^{q-2}(q-1) \left|\int_{|z|\geq \delta_1} (|u(y) - u(y+z)| + |u(y)-\chi_m(y+z)|)^{q-2} |u(y+z)-\chi_m(y+z)|\frac{\dz}{|z|^{n+sq}}\right|
\\
&\leq C \left|\int_{\Rn} (|u(y)| + |u(y+z)| + |\chi_m(y+z)|)^{q-2} |u(y+z)-\chi_m(y+z)|\frac{\dz}{1+|y+z|^{n+sq}}\right|
\\
&\leq C \left[\int_{\Rn} \frac{(1+ |u(z)|^{q-1}+|\chi_m(z)|^{q-1})}{1+|z|^{n+sq}}\dz\right]^{\frac{q-2}{q-1}}
\left[\int_{\Rn} \frac{ |u(z)-\chi_m(z)|^{q-1}}{1+|z|^{n+sq}}\dz\right]^{\frac{1}{q-1}}\to 0,
\end{align*}
as $m\to \infty$ uniformly in $\mathcal{K}$, where $C$ can be chosen independent of $y$, since $|u(y)|\leq \kappa$ and
$$\inf_{|z|\geq\delta_1} \frac{|z|^{n+sq}}{1+|z+y|^{n+sq}}>0, \quad \mbox{uniformly in $\mathcal{K}$}.$$
Analogously, for $q\in (1, 2]$ for $y\in \mathcal{K}$,
\begin{align*}
&\left|\int_{|z|\geq \delta_1} J_q(u(x_k)-u(x_k+z))\frac{\dz}{|z|^{n+sq}}-\int_{|z|\geq \delta_1} J_q(u(x_k)-\chi_m(x_k+z))\frac{\dz}{|z|^{n+sq}}\right|
\\
&\leq C \left|\int_{\Rn} |u(x_k+z)-\chi_m(x_k+z)|^{q-1}\frac{\dz}{1+|z+x_k|^{n+sq}}\right|
\\
&= C \left|\int_{\Rn} |u(z)-\chi_m(z)|^{q-1}\frac{\dz}{1+|z|^{n+sq}}\right|\to 0, 
\end{align*}
as $m\to\infty$ uniformly in $\mathcal{K}$. Now the proof follows using the fact that for every fixed $m$ we have
$$\int_{|z|\geq \delta_1} J_q(u(x_k)-\chi_m(x_k+z))\frac{\dz}{|z|^{n+qs}}\to \int_{|z|\geq \delta_1} J_q(u(\xbar)-\chi_m(\xbar+z))\frac{\dz}{|z|^{n+qs}},$$
as $k\to \infty$. 
\end{proof}
We also need the following convergence result.
\begin{lem}\label{C2conv}
Suppose that $\psi_k,\psi \in C^2(\bar{B}_r(x_0))$ for some $r>0$, $x_0\in\Rn$ and consider a 
sequence of points $x_k\to x_0$. Also, assume that $\psi_k\to\psi$ in $C^2(\bar{B}_r(x_0))$
as $k\to\infty$. If $\grad\psi(x_0)\neq 0$, then we have
$$\lim_{k\to\infty}\sL[B_{r_1}(x_k)]\psi_k(x_k)= \sL[B_{r_1}(x_0)]\psi(x_0)$$
for any $r_1<r$.
\end{lem}

\begin{proof}
We write
\begin{align*}
\sL[B_{r_1}(x_k)]\psi_k(x_k)&={\rm PV}\int_{|z|<r_1} J_q(\psi_k(x_k)-\psi_k(x_k+z))\frac{\dz}{|z|^{n+sq}},
\\
\sL[B_{r_1}(x_0)]\psi(x_0)&={\rm PV}\int_{|z|<r_1} J_q(\psi(x_0)-\psi(x_0+z))\frac{\dz}{|z|^{n+sq}}.
\end{align*}
Since $|\grad\psi(x_0)|>0$ and $x_k\to x_0$, from our assertion, we can find $r_2<r_1$ such that $|\grad\psi_k|, |\grad \psi|>0$ in $B_{r_2}(x_0)$
for all $k$ large.
Therefore, by \cite[Lemma~3.6]{KKL19}, given $\varepsilon>0$ there exists $\delta_\varepsilon<r_2$ satisfying
$$ \left|{\rm PV}\int_{|z|<\delta_\varepsilon} J_q(\psi_k(x_k)-\psi_k(x_k+z))\frac{\dz}{|z|^{n+sq}}\right|
+ \left|{\rm PV}\int_{|z|<\delta_\varepsilon} J_q(\psi_k(x_k)-\psi_k(x_k+z))\frac{\dz}{|z|^{n+sq}}\right|<\varepsilon$$
for all $k$ large. Again, by the dominated convergence theorem, we have 
$$\lim_{k\to\infty}\int_{\delta_\varepsilon\leq |z|\leq r_1} J_q(\psi_k(x_k)-\psi_k(x_k+z))\frac{\dz}{|z|^{n+sq}}
=\int_{\delta_\varepsilon\leq |z|\leq r_1}  J_q(\psi(x_0)-\psi(x_0+z))\frac{\dz}{|z|^{n+sq}}.$$
Thus, combining the above displays, we have the result.
\end{proof}

\section{Local almost Lipschitz regularity}\label{S-lip}
For the proofs of this section, we assume that $u\in C(\Omega)\cap L^{q-1}_{sp}(\Rn)$ is a viscosity solution to
\begin{equation}\label{E2.1}
-\norm{f}_\infty\leq \sL u\leq \norm{f}_\infty\quad \text{in}\; \Omega
\end{equation}
at the non-critical points, in the sense of Proposition~\ref{P1.2}. This is a valid setting in view of Theorem~\ref{Thm-cont}.
Our main result of this section is the almost Lipschitz regularity.
\begin{thm}\label{T-main1}
Let $u\in C(\Omega)\cap L^{q-1}_{sq}(\Rn)$ be a solution to \eqref{E2.1} in the viscosity sense, as mentioned in Proposition~\ref{P1.2}.
Then for any $\tilde\Omega\Subset \Omega$, we have, for any $\beta\in (0, 1)$, that
$$\norm{u}_{C^{0, \beta}(\tilde\Omega)}\leq \tilde{C}, $$
where the constant $\tilde{C}$ depends on $\data, \beta$ and $\dist(\tilde\Omega, \partial\Omega)$.
\end{thm}
It is not difficult to see that we can always assume $u$ to globally continuous and bounded. To see this,
consider $\tilde\Omega\Subset\Omega_1\Subset\Omega_2\Subset\Omega_3\Subset\Omega$. Let
$\chi:\Rn\to [0, 1]$ be a smooth cut-off function satisfying $\chi=1$ in $\Omega_2$ and $\chi=0$ on $\Omega^c_3$. Letting, 
$w=\chi u$, it is easy to see from \eqref{E2.1} that
\begin{equation}\label{E2.2}
-C \leq \sL w\leq C \quad \text{in}\; \Omega_1,\quad \text{where}\quad C= \norm{f}_\infty + \kappa
\left(\sup_{\Omega_1} |u|^{q-1}+ \int_{\Rn}\frac{|u(z)|^{q-1}}{1+|z|^{n+sq}}\dz\right)
\end{equation}
for some constant $\kappa$, dependent of $\dist(\Omega_1, \Omega^c_2)$. To see this, we note that for $x\in\Omega_1$ we can write
, for $q>2$,
\begin{align*}
&|\sL_q u(x) - \sL_q w(x)| 
\\&\leq (q-1) 2^{q-2} \int_{\Rn} (|u(x+z)-u(x)|+|w(x+z)-u(x)|)^{q-2} |(1-\chi(x+z)) u(x+z)|\frac{\dz}{|z|^{n+sq}}
\\
&= (q-1)2^{q-2}\int_{|z|\geq \frac{1}{2}\dist(\Omega_1, \Omega^c_2)} (|u(x+z)-u(x)|+|w(x+z)-u(x)|)^{q-2} |(1-\chi(x+z)) u(x+z)|\frac{\dz}{|z|^{n+sq}}
\\
&\leq \kappa \int_{\Rn} (|u(x)|^{q-1}+|u(z)|^{q-1}) \frac{\dz}{1+|z|^{n+sq}}.
\end{align*}
A similar estimate also holds for $q\in (1, 2]$, giving us \eqref{E2.2}. Therefore, in view of \eqref{E2.1} and \eqref{E2.2}, it is enough to investigate the
situation where $u\in C(\Rn)$ is globally bounded and 
\begin{equation}\label{E2.3}
-C\leq \sL u\leq C\quad  \text{in}\; \Omega,
\end{equation}
in the viscosity sense and at the non-critical points. We consider two concentric balls $B\Subset \tilde{B}\Subset \Omega$. For the economy of notation, we assume that
$B=B_1(0)=B_1$ and $\tilde{B}=B_2(0)=B_2$. 
Fix $1\leq \varrho_1<\varrho_2\leq 2$, and define the doubling
function
\begin{equation}\label{E2.4}
\Phi(x, y)= u(x)-u(y)-L\varphi(|x-y|)- m_1 \psi(x)\quad x, y\in \Rn,
\end{equation}
where 
$$
\psi(x) = [(|x|^2-\varrho^2_1)_+]^{m}, \ x \in \Rn,
$$ 
is a {\it localization} function. We set $m\geq 3$ so that $\psi\in C^2(B_2)$. The function $\varphi:[0, \infty)\to [0, \infty)$ is a {\it regularizing
function} given by $\varphi(t)=t^\gamma$ for $\gamma\in (0, 1)$. We set $m_1$ large enough so that
$$m_1 \psi(x)\geq 2\,\sup_{\Rn}|u|\quad \text{for}\;\; |x|\geq \frac{\varrho_1+\varrho_2}{2}.$$
Our primary goal of this section is to show that there exists 
$L$ large enough, but independent of $u$, so that $\Phi\leq 0$ in $\Rn\times \Rn$. Note that this leads to $\gamma$-H\"{o}lder estimate of
$u$ in $B_{\varrho_1}$ with H\"{o}lder constant $\gamma$. Then we repeat this estimate in smaller ball to improve the H\"{o}lder
exponent $\gamma$, leading to an
almost Lipschitz estimate. 

We suppose, on the contrary, that $\Phi\nleq 0$ in $\Rn$ for all large $L$, which implies that $\sup_{\Rn\times\Rn}\Phi>0$.
By our choice of $m_1$, we have 
$\Phi(x, y)<0$ for all $y\in \Rn$ and $|x|\geq \frac{\varrho_2+\varrho_1}{2}$. Again, since $\varphi$ is strictly increasing in
$[0, 2]$, if we choose $L$ to satisfy $L\varphi(\frac{\varrho_2-\varrho_1}{4})>2\sup_{\Rn}|u|$, we obtain $\Phi(x, y)<0$ whenever
$|x-y|\geq \frac{\varrho_2-\varrho_1}{4}$. Thus, there exists $\xbar\in B_{\frac{\varrho_2+\varrho_1}{2}}$ and 
$\ybar\in B_{\frac{3\varrho_2}{4}+\frac{\varrho_1}{4}}$ such that
\begin{equation}\label{E2.5}
\sup_{\Rn\times \Rn}\Phi=\Phi(\xbar,\ybar)>0.
\end{equation}
Denote by $\abar=\xbar-\ybar$. From \eqref{E2.5} we have $\abar\neq 0$, and moreover, we have that
\begin{equation}\label{E2.6}
L \varphi(|\bar a|) \leq u(\bar x) - u(\bar y) \leq 2 \sup_{\Rn}|u|.
\end{equation}
This implies that $|\bar a|$ gets smaller as $L$ enlarges. Also, denote by
\begin{gather*}
\phi(x, y)= L\varphi(|x-y|)+ m_1 \psi(x), \quad \abar=\xbar-\ybar, \quad \bar\xi_x= \grad_x\phi(\xbar,\ybar)=L\varphi'(|\abar|)\frac{\abar}{|\abar|} + m_1 \grad\psi(\xbar),
\\
\text{and}\quad \bar\xi_y=\grad_y\phi(\xbar,\ybar)=-L\varphi'(|\abar|)\frac{\abar}{|\abar|}.
\end{gather*}
Since
$$|\bar\xi_x|\geq L\gamma |\abar|^{\gamma-1}- m_1\max_{B_2}|\grad \psi|, $$
using \eqref{E2.5} we can choose
 $L_0$ large enough, dependent on $m, m_1, \gamma$ and $\varrho_1$, so that $\bar\xi_x\neq 0$ and $\bar\xi_y\neq 0$ for all $L\geq L_0$. 
Again, for any $\delta<|\abar|/2$ we have $\phi\in C^2(B_\delta(\xbar,\ybar))$. At this point we invoke nonlocal Jensen-Ishii lemma given by
Lemma~\ref{LJI}. Fix $\delta_1=\frac{1}{2}\delta$ and choose $\phi_k, x_k, y_k$ from Lemma~\ref{LJI}. Note that
\begin{align*}
x &\mapsto u(x)-\phi_k(x, y_k)\quad \mbox{has a global maximum at $x_k$, and}\\
y &\mapsto u(y)+\phi_k(x_k, y)\quad \mbox{has a global minimum at $y_k$}.
\end{align*}
So we define
\[
w_k(x)=\left\{\begin{array}{ll}
\phi_k(x, y_k) & \text{for}\; x\in B_{\frac{\delta_1}{2}}(x_k),
\\[2mm]
u(x) & \text{otherwise},
\end{array}
\right.
\]
and
\[
\tilde{w}_k(y)=\left\{\begin{array}{ll}
-\phi_k(x_k, y) & \text{for}\; y\in B_{\frac{\delta_1}{2}}(y_k),
\\[2mm]
u(y) & \text{otherwise}.
\end{array}
\right.
\]
Since $\bar\xi_x\neq 0$ and $\bar\xi_y\neq 0$ for $L\geq L_0$, for large enough $k$ we would have $p_k\neq 0$ and $q_k\neq 0$. Therefore,
we can apply Proposition~\ref{P1.3} to obtain from \eqref{E2.3} that
\begin{equation}\label{E2.7}
-F(p_k, X_k) + \sL_q w(x_k)\leq C \quad \text{and}\quad -F(-q_k, Y_k) + \sL_q \tilde w(y_k)\geq -C.
\end{equation}
Since $\phi_k\to \phi_\alpha$ in $C^2(B_{\delta_1}(\xbar, \ybar))$, as $k\to\infty$, from Lemma~\ref{LJI}(4), we can find $X_\alpha, Y_\alpha \in\mathbb{S}^n$ satisfying
$$
\begin{pmatrix}
X_\alpha & 0\\
0 & -Y_\alpha
\end{pmatrix}
\leq D^2\phi_\alpha(\xbar, \ybar)= D^2\phi (\xbar, \ybar) + o_\alpha(1),
$$
and
$$ (X_k, Y_k)\to (X_\alpha, Y_\alpha),\quad F(p_k, X_k)\to F(\bar\xi_{x}, X_\alpha), \quad \text{and}\quad F(-q_k, Y_k)\to F(-\bar\xi_{y}, Y_\alpha),$$
possibly along some subsequence. From Lemma~\ref{lemcontnonlocal} and ~\ref{C2conv} we also get
$$\sL_q w_k(x_k)\to \sL_q w_\alpha(\xbar)\quad\text{and}\quad \sL_q\tilde{w}_k(y_k)\to \sL_q \tilde{w}_\alpha(\ybar),$$
where $w_\alpha$ and $\tilde{w}_\alpha$ are defined in an analogous fashion as $w_k$ and $\tilde{w}_k$, respectively, 
with $\phi_k(\cdot, y_k)$ and $\phi_k(x_k, \cdot)$ being replaced by $\phi_\alpha(\cdot,\ybar)$ and $\phi_\alpha(\xbar, \cdot)$.
Thus we obtain from \eqref{E2.7} that
\begin{equation}\label{E2.8}
\underbrace{-F(\bar\xi_{x}, X_\alpha) + F(-\bar\xi_{y}, Y_\alpha)}_{=\mathcal A_\alpha} + 
\sL_q w_\alpha(\xbar)-\sL_q \tilde w_\alpha (\ybar)\leq 2C.
\end{equation}
Our next step would be to send $\alpha\to 0$ in the above expression, but we need to estimate the term 
$\mathcal{A_\alpha}$ first, uniformly in $\alpha$.

\subsection{Estimation of $\mathcal{A}_\alpha$}
Recall that $\bar\xi_y\neq 0$. Denote by $\hat\xi_y=\bar\xi_y/|\bar\xi_y|$, the unit vector along $\bar\xi_y$. Now pick a set of orthonormal
vectors ${\nu_1, \ldots, \nu_{n-1}}$ so that $(\hat\xi_y, \nu_1,\ldots, \nu_{n-1})$ form an orthonormal basis in $\Rn$.
Let $\hat\xi_x=\bar\xi_x/|\bar\xi_x|$. Note that if we choose $L\geq L_0$ large enough, depending on $m, m_1$ and $\norm{u}_\infty$,  we have
$$\langle \hat\xi_x, \hat\xi_y\rangle\geq 
\frac{L\varphi'(\abar)(L\varphi'(\abar)-  m_1 |\grad\psi(\xbar)|)}{L\varphi'(\abar)(L\varphi'(\abar)+  m_1 |\grad\psi(\xbar)|)} 
\geq \frac{1}{\sqrt{2}}.$$
Thus, $(\hat\xi_x, \nu_1,\ldots, \nu_{n-1})$ are independent and form a basis of $\Rn$. Let $(\hat\xi_x, \tilde\nu_1,\ldots, \tilde\nu_{n-1})$
denote the orthonormal basis obtained from $(\hat\xi_x, \nu_1,\ldots, \nu_{n-1})$ by the Gram-Schmidt process. More precisely,
\begin{align*}
\breve\nu_1 &= \nu_1 - \langle \hat\xi_x, \nu_1\rangle \hat\xi_x, \quad \tilde\nu_1=\frac{\breve\nu_1}{|\breve\nu_1|},
\\
\breve\nu_i & = \nu_i - \left(\langle \hat\xi_x, \nu_i\rangle \hat\xi_x + \sum_{j=1}^{ i-1} \langle\tilde\nu_j, \nu_i\rangle \tilde\nu_j\right), \quad \tilde\nu_i=\frac{\breve\nu_i}{|\breve\nu_i|},
\end{align*}
for $i=2, \ldots, n-1$. Also, denote by $\rho(\xbar)=\frac{m_1}{L\varphi'(|\abar|)}|\grad\psi(\xbar)|$.

\begin{lem}\label{L2.1}
There exists a constant $\kappa$, dependent only on $n$, such that
$$ |\hat\xi_x-\hat\xi_y|\leq \kappa\, \sqrt{\rho(\xbar)}, \quad |\nu_i-\tilde\nu_i|\leq \kappa\, \sqrt{\rho(\xbar)}\quad \text{for}\; i=1, 2, \ldots, n-1,$$
for all $L\geq L_0$, where $L_0$ is fixed at some large value depending on $m, m_1$ and $\norm{u}_\infty$, but not on $\xbar$.
\end{lem}

\begin{proof}
Observe that, due to our choice of $\varphi$ and \eqref{E2.6}, $\rho\to 0$ as $L\to\infty$ (uniformly in $x\in B_2$).
We first note that 
$$\langle \hat\xi_x, \hat\xi_y\rangle\geq 
\frac{L\varphi'(\abar)(L\varphi'(\abar)-  m_1 |\grad\psi(\xbar)|)}{L\varphi'(\abar)(L\varphi'(\abar)+  m_1 |\grad\psi(\xbar)|)}
\geq \frac{1-\rho(\xbar)}{1+\rho(\xbar)}.$$
Thus, 
$$|\hat\xi_x-\hat\xi_y|^2 =2(1-\langle \hat\xi_x, \hat\xi_y\rangle)\leq \frac{4\rho(\xbar)}{1+\rho(\xbar)}\leq 4\rho(\xbar)\Rightarrow |\hat\xi_x-\hat\xi_y|\leq 2\sqrt{\rho(\xbar)}.$$
Now continue the proof by the method of induction.
Suppose that $|\nu_i-\tilde\nu_i|\leq \kappa \sqrt{\rho(\xbar)}$ for some $\kappa$ and for $i=1,2\ldots, k-1$. 
Then
\begin{align*}
|\breve\nu_k| &\leq 1 + \left(|\langle \hat\xi_x, \nu_i\rangle|  + \sum_{j=1}^{ k-1} |\langle\tilde\nu_j, \nu_i\rangle|\right)
\\
&\leq 1 + \left(|\langle \hat\xi_x-\hat\xi_y, \nu_i\rangle|  + \sum_{j=1}^{ k-1} |\langle\tilde\nu_j-\nu_j, \nu_i\rangle|\right)
\\
&\leq 1 + \kappa k \sqrt{\rho(\xbar)}.
\end{align*}
Similarly, we also have  $|\breve\nu_k|\geq 1-\kappa\, k \sqrt{\rho(\xbar)}$.  Choosing $L$ large we obtain that
$$ \left| 1-|\breve\nu_k|^{-1}\right|\leq \frac{\kappa k \sqrt{\rho(\xbar)}}{1-\kappa k \sqrt{\rho(\xbar)}}\leq 2\kappa k \sqrt{\rho(\xbar)}.$$
Now, from the definition we obtain
\begin{align*}
|\nu_k-\tilde\nu_k| &\leq | 1-|\breve\nu_k|^{-1}| +  \frac{1}{|\breve\nu_k|} \left(|\langle \hat\xi_x, \nu_i\rangle|  + \sum_{j=1}^{ k-1} |\langle\tilde\nu_j, \nu_i\rangle|\right)
\\
&\leq  2\kappa k \sqrt{\rho(\xbar)}  +  \frac{1}{|\breve\nu_k|} \left(|\langle \hat\xi_x-\hat\xi_y, \nu_i\rangle|  + \sum_{j=1}^{ k-1} |\langle\tilde\nu_j-\nu_j, \nu_i\rangle|\right)
\\
& \leq 2\kappa k \sqrt{\rho(\xbar)} + \frac{\kappa k\sqrt{\rho(\xbar)}}{1-\kappa k\sqrt{\rho(\xbar)}}\leq 4\kappa k\sqrt{\rho(\xbar)}.
\end{align*}
Thus, replacing $\kappa$ by $4\kappa k$ we have the estimate for $i=1, 2, \ldots, k$. This completes the proof.
\end{proof}

For the next lemma we recall that
\begin{equation}\label{E2.9}
\begin{pmatrix}
X_\alpha & 0\\
0 & -Y_\alpha
\end{pmatrix}
\leq D^2\phi_\alpha(\xbar, \ybar)= D^2\phi (\xbar, \ybar) + o_\alpha(1),
\end{equation}
which is a consequence of the fact that $\phi_\alpha \to \phi$ in $C^2(\bar{B}_{\delta}(\xbar,\ybar))$, as $\alpha\to 0$.
\begin{lem}\label{L2.2}
Denote by $\sM_x=D^2_{xx}\phi_\alpha(\xbar,\ybar)$, $\sM_y=D^2_{yy}\phi_\alpha(\xbar,\ybar)$ and $\zeta_i=(\tilde\nu_i, \nu_i)$ for $i=1, 2,\ldots, n-1$. Then
\begin{align}\label{EL2.2A}
F(\bar\xi_x, X_\alpha)-F(-\bar\xi_y, Y_\alpha) &\leq (p-1) \left[|\bar\xi_x|^{p-4}\langle \bar\xi_x \sM_x, \bar\xi_x\rangle + |\bar\xi_y|^{p-4}\langle \bar\xi_y \sM_y, \bar\xi_y\rangle\right]\nonumber
\\
&\quad + (|\bar\xi_x|^{p-2}+|\bar\xi_y|^{p-2}) \sum_{i=1}^{n-1}\langle \zeta_i D^2\phi_\alpha(\xbar,\ybar), \zeta_i\rangle_+ \nonumber
\\
&\qquad + \left||\bar\xi_x|^{p-2}- |\bar\xi_y|^{p-2}\right|\sum_{i=1}^{n-1} \left(|\langle \tilde\nu_i \sM_x, \tilde\nu_i\rangle| + |\langle \nu_i \sM_y, \nu_i\rangle|\right).
\end{align}
\end{lem}

\begin{proof}
From \eqref{E2.9} we observe that $\langle \zeta X_\alpha, \zeta\rangle \leq \langle \zeta \sM_x, \zeta\rangle$ and $\langle \zeta Y_\alpha, \zeta\rangle\geq -\langle \zeta \sM_y, \zeta\rangle$
for all $\zeta\in\Rn$. First, we suppose that $|\bar\xi_x|^{p-2}\leq |\bar\xi_y|^{p-2}$. We write using \eqref{defF}
\begin{align*}
F(\bar\xi_x, X_\alpha)-F(-\bar\xi_y, Y_\alpha) &=
|\bar\xi_x|^{p-2} \Bigl(\langle \hat\xi_x X_\alpha, \hat\xi_x\rangle + \sum_{i=1}^{n-1} \langle \tilde\nu_i X_\alpha, \tilde\nu_i\rangle\Bigr) + (p-2)|\bar\xi_x|^{p-4} \langle \bar\xi_x X_\alpha, \bar\xi_x\rangle
\\
&\quad - |\bar\xi_y|^{p-2} \Bigl(\langle \hat\xi_y Y_\alpha, \hat\xi_y\rangle + \sum_{i=1}^{n-1} \langle \nu_i Y_\alpha, \nu_i\rangle\Bigr) - (p-2)|\bar\xi_y|^{p-4} \langle \bar\xi_y Y_\alpha, \bar\xi_y\rangle
\\
&= (p-1) |\bar\xi_x|^{p-4} \langle \bar\xi_x X_\alpha, \bar\xi_x\rangle - (p-1) |\bar\xi_y|^{p-4} \langle \bar\xi_y Y_\alpha, \bar\xi_y\rangle
\\
&\quad - (|\bar\xi_y|^{p-2}-|\bar\xi_x|^{p-2}) \sum_{i=1}^{n-1} \langle \nu_i Y_\alpha, \nu_i\rangle
 + |\bar\xi_x|^{p-2} \sum_{i=1}^{n-1} \left(\langle \tilde\nu_i X_\alpha, \tilde\nu_i\rangle- \langle \nu_i Y_\alpha, \nu_i\rangle \right)
 \\
& \leq  (p-1) |\bar\xi_x|^{p-4} \langle \bar\xi_x \sM_x, \bar\xi_x\rangle + (p-1) |\bar\xi_y|^{p-4} \langle \bar\xi_y \sM_y, \bar\xi_y\rangle
\\
&\quad + (|\bar\xi_y|^{p-2}-|\bar\xi_x|^{p-2}) \sum_{i=1}^{n-1} \langle \nu_i \sM_y, \nu_i\rangle 
+ |\bar\xi_x|^{p-2} \sum_{i=1}^{n-1} \langle \zeta_i D^2\phi_\alpha(\xbar,\ybar), \zeta_i\rangle
\\
& \leq  (p-1) |\bar\xi_x|^{p-4} \langle \bar\xi_x \sM_x, \bar\xi_x\rangle + (p-1) |\bar\xi_y|^{p-4} \langle \bar\xi_y \sM_y, \bar\xi_y\rangle
\\
&\quad + \left| |\bar\xi_y|^{p-2}-|\bar\xi_x|^{p-2}\right| \sum_{i=1}^{n-1} |\langle \nu_i \sM_y, \nu_i\rangle|
+ |\bar\xi_x|^{p-2} \sum_{i=1}^{n-1} \langle \zeta_i D^2\phi_\alpha(\xbar,\ybar), \zeta_i\rangle_+,
\end{align*}
where in the third line we use \eqref{E2.9} by multiplying both sides of the matrices with $\zeta_i$.
A similar calculation also holds when $|\bar\xi_x|^{p-2}\geq |\bar\xi_y|^{p-2}$, giving us \eqref{EL2.2A}.
\end{proof}

Letting $\alpha\to 0$ and using $\phi_\alpha\to \phi$ in $C^2(\bar{B}_{\delta}(\xbar, \ybar))$, we get 
\begin{lem}\label{L2.3}
There exist a constant $C_1$, dependent on $p$, and a constant
$ L_0$, dependent on $n, m, m_1, \norm{u}_\infty$, such that
\begin{equation}\label{EL2.3A}
\limsup_{\alpha\to 0} (F(\bar\xi_x, X_\alpha)-F(-\bar\xi_y, Y_\alpha))\leq  C_1\, L^{p-1} (\varphi'(|\abar|))^{p-2}\varphi^{\prime\prime}(|\abar|)
\end{equation}
for all $L\geq L_0$ and $\varphi(t)=t^\gamma$ for $\gamma\in (0, 1)$. 
%If $u$ in H\"{o}lder continuous in $\bar{B_2}$, similar estimate hold
%for the Lipschitz profile function in \eqref{thevarphis}.
\end{lem}

\begin{proof}
First, we observe that the vectors $(\hat\xi_x, \tilde\nu_1, \ldots, \tilde\nu_{n-1})$, $(\hat\xi_y, \nu_1, \ldots, \nu_{n-1})$ do
not depend on $\alpha$. Since $\phi_\alpha\to \phi$ in $C^2(\bar{B}_{\delta}(\xbar, \ybar))$, letting $\alpha\to 0$ in \eqref{EL2.2A} we obtain
\begin{align}\label{EL2.3AA}
\limsup_{\alpha\to 0}(F(\bar\xi_x, X_\alpha)-F(-\bar\xi_y, Y_\alpha)) &\leq (p-1) \left[|\bar\xi_x|^{p-4}\langle \bar\xi_x \tilde\sM_x, \bar\xi_x\rangle +
 |\bar\xi_y|^{p-4}\langle \bar\xi_y \tilde\sM_y, \bar\xi_y\rangle\right]\nonumber
\\
&\quad + (|\bar\xi_x|^{p-2}+|\bar\xi_y|^{p-2}) \sum_{i=1}^{n-1}\langle \zeta_i D^2\phi(\xbar,\ybar), \zeta_i\rangle_+ \nonumber
\\
&\qquad + \left||\bar\xi_x|^{p-2}- |\bar\xi_y|^{p-2}\right|\sum_{i=1}^{n-1} \left(|\langle \tilde\nu_i \tilde\sM_x, \tilde\nu_i\rangle| + 
|\langle \nu_i \tilde\sM_y, \nu_i\rangle|\right),
\end{align}
where $\tilde\sM_x=D^2_{xx}\phi(\xbar, \ybar)$ and $\tilde\sM_y=D^2_{yy}\phi(\xbar, \ybar)$. Recall that $\phi(x, y)= L\varphi(|x-y|) + m_1\psi(x)$.
We denote by 
$$M= L \varphi^{\prime\prime}(|\abar|)\frac{\abar\otimes\abar}{|\abar|^2}  + L\frac{\varphi'(|\abar|)}{|\abar|}\left(I-\frac{\abar\otimes\abar}{|\abar|^2},
\right)$$
and $N=m_1 D^2\psi(\xbar)$. It is easily seen that
\begin{equation*}
D^2\phi(\xbar, \ybar)=
\begin{pmatrix}
M & -M\\
-M & M
\end{pmatrix} +
\begin{pmatrix}
N & 0\\
0 & 0
\end{pmatrix},
\end{equation*}
$\tilde\sM_x = M+N$ and $\tilde\sM_y=M$. Now we begin with the estimate of $\langle \bar\xi_x \tilde\sM_x, \bar\xi_x\rangle$. Recalling
$\bar\xi_x=L\varphi'(|\abar|)\frac{\abar}{|\abar|} + m_1 \grad\psi(\xbar)$, we have
\begin{align*}
\langle \bar\xi_x \tilde\sM_x, \bar\xi_x\rangle &= \langle \bar\xi_x (M+N), \bar\xi_x\rangle
\\
&\leq L^3 (\varphi'(|\abar|))^2\varphi^{\prime\prime}(|\abar|) + 2m_1|L\varphi'(|\abar|)||\grad\psi(\xbar)||M|
\\
&\qquad + m_1^2 |\grad\psi(\xbar)|^2 |M| + |\bar\xi_x|^2 m_1 |D^2\psi(\xbar)|
\\
&\leq L^2(\varphi'(|\abar|))^2[L\varphi^{\prime\prime}(|\abar|) + (2\rho(\xbar) +\rho^2(\xbar)) |M|] + |\bar\xi_x|^2 m_1 |D^2\psi(\xbar)|,
\end{align*}
where $\rho(\xbar)=\frac{m_1|\grad\psi(\xbar)|}{L\varphi'(|\abar|)}$. By the choice of $\varphi$, 
we have $|M|\leq \kappa_1 L \frac{\varphi'(|\abar|)}{|\abar|}=\frac{\kappa_1}{1-\gamma}L|\varphi^{\prime\prime}(\abar)|$ for some $\kappa_1>0$.
Since $\rho(\xbar)\to 0$ as $L\to\infty$ (uniformly in $\xbar$), for $\varphi(t)=t^{\gamma}$, it follows that
$$L\varphi^{\prime\prime}(|\abar|) + (2\rho(\xbar) +\rho^2(\xbar)) |M|\leq \frac{1}{2}L \varphi^{\prime\prime}(|\abar|)$$
for all $L$ large. 
%If $u\in C^\upkappa(\bar{B}_2)$, then since $\Phi(\xbar, \ybar)>0$, we get
%$$m_1\psi(\xbar)\leq u(\xbar)-u(\ybar)\leq C|\abar|^\upkappa.$$
%Thus, $(|\xbar|^2-\varrho_1)_+\leq C |\abar|^{\frac{\upkappa}{m}}$, implying $|\grad\psi(\xbar)|\leq C |\abar|^{\frac{\upkappa(m-1)}{m}}$,
%where the constant $C$ depends on $m_1$.
%Then, for the Lipschitz profile function, it can be easily seen that
%$$ L\varphi^{\prime\prime}(|\abar|) + 2\rho(x)|M| + (\rho(x))^2 |M|
%\leq \frac{1}{2}L \varphi^{\prime\prime}(|\abar|)
%$$
%for all $L$ large. 
Since $\frac{L}{2}\varphi'(|\abar|)\leq |\bar\xi_x|\leq 2L\varphi'(|\abar|)$ for all large enough $L$, we obtain 
\begin{equation}\label{EL2.3B}
|\bar\xi_x|^{p-4}\langle \bar\xi_x \tilde\sM_x, \bar\xi_x\rangle\leq \kappa_2 L^{p-1} (\varphi'(|\abar|))^{p-2} \varphi^{\prime\prime}(|\abar|)
\end{equation}
for all large $L$ and some constant $\kappa_2$. Letting $\psi=0$, a similar estimate also holds for $ |\bar\xi_y|^{p-4}\langle \bar\xi_y \tilde\sM_y, \bar\xi_y\rangle$.

To estimate $\langle \zeta_i D^2\phi(\xbar,\ybar), \zeta_i\rangle_+$ for $i=1, 2, \ldots, n-1$, we compute 
\begin{align*}
\langle \zeta_i D^2\phi(\xbar,\ybar), \zeta_i\rangle &\leq \langle \tilde\nu_i M, \tilde\nu_i-\nu_i \rangle + \langle (\nu_i-\tilde\nu_i) M,  \nu_i\rangle + m_1|D^2\psi(\xbar)|
\\
&\leq 2|M| |\tilde\nu_i-\nu| + m_1|D^2\psi(\xbar)|
\\
&\leq \kappa_3 L\frac{\varphi'(|\abar|)}{|\abar|} \sqrt{\rho(\xbar)} + m_1|D^2\psi(\xbar)|,
\end{align*}
for some constant $\kappa_3$, using Lemma~\ref{L2.1}. Arguing as above, we can choose $L$ large enough so that 
\begin{equation}\label{EL2.3C}
(|\bar\xi_x|^{p-2}+|\bar\xi_y|^{p-2})\sum_{i=1}^{n-1}\langle \zeta_i D^2\phi(\xbar,\ybar), \zeta_i\rangle_+\leq - \frac{\kappa_2(p-1)}{4} L^{p-1} (\varphi'(|\abar|))^{p-2} \varphi^{\prime\prime}(|\abar|).
\end{equation}
To compute the last term in \eqref{EL2.3AA} we observe that
\begin{align*}
\left||\bar\xi_x|^{p-2}- |\bar\xi_y|^{p-2}\right|= |\bar\xi_y|^{p-2} \left|1-\left|1+\frac{m_1\grad\psi(\xbar)}{L\varphi'(|\abar|)}\right|^{p-2}\right|
\leq \kappa_4 |L\varphi'(|\abar|)|^{p-2} \rho(\xbar),
\end{align*}
using Lipschitz property of the map $x\mapsto |x|^{p-2}$ around $|x|=1$. Thus, for large enough $L$ we obtain
\begin{equation}\label{EL2.3D}
\left||\bar\xi_x|^{p-2}- |\bar\xi_y|^{p-2}\right|\sum_{i=1}^{n-1} \left(|\langle \tilde\nu_i \tilde\sM_x, \tilde\nu_i\rangle| + 
|\langle \nu_i \tilde\sM_y, \nu_i\rangle|\right)\leq - \frac{\kappa_2(p-1)}{4} L^{p-1} (\varphi'(|\abar|))^{p-2} \varphi^{\prime\prime}(|\abar|).
\end{equation}
Hence \eqref{EL2.3A} follows combining \eqref{EL2.3B}, \eqref{EL2.3C} and \eqref{EL2.3D}.
\end{proof}
To this end, we define $\delta_1=\frac{1}{2}\delta<\frac{1}{4}|\abar|$
\begin{equation*}
w(x)=\left\{\begin{array}{ll}
\phi(x, \ybar) & \text{for}\; x\in B_{\frac{\delta_1}{2}}(\xbar),
\\[2mm]
u(x) & \text{otherwise},
\end{array}
\right.
\text{and}\quad 
\tilde{w}(y)=\left\{\begin{array}{ll}
-\phi(\xbar, y) & \text{for}\; y\in B_{\frac{\delta_1}{2}}(\ybar),
\\[2mm]
u(y) & \text{otherwise}.
\end{array}
\right.
\end{equation*}
Since $\phi_\alpha \to \phi$ in $C^2(\bar{B}_{\delta}(\xbar, \ybar))$, letting $\alpha\to 0$ in \eqref{E2.8}, using Lemmas~\ref{lemcontnonlocal},~\ref{C2conv}
and ~\ref{L2.3}, we arrive at
\begin{equation}\label{E2.16}
-C_1 L^{p-1} (\varphi'(|\abar|))^{p-2}\varphi^{\prime\prime}(|\abar|) + 
\underbrace{\sL_q w(\xbar)-\sL_q \tilde w (\ybar)}_{=\mathcal{B}}\leq 2C,
\end{equation}
provided $L\geq L_0$, where $L_0$ is given by Lemma~\ref{L2.3}. In the remaining part of this section we estimate $\mathcal{B}$ suitably
to draw a contradiction to \eqref{E2.16}. Towards this goal, we introduce the notation
$$\tilde\sL_q[D] u(x):= \sL_q[x+D] u(x)={\rm PV}\int_D J_q(u(x)-u(x+z))\frac{\dz}{|z|^{n+sq}},$$
where $D$ is a measurable set. Let $\delta_0\in (0, \frac{1}{8})$, $\tilde\varrho=\frac{1}{4}(\varrho_2-\varrho_1)$. We also define the following
sets
$$\cone=\{z\in B_{\delta_0|\abar|}\;:\; |\langle \abar, z\rangle|\geq (1-\delta_0) |\abar| |z| \},
\quad \cD_1=B_{\tilde\delta}\cap \cone^c\quad \cD_2=B_{\tilde\varrho}\setminus (\cone\cup\cD_1).$$
Later we are going to set $\delta_0, \tilde\delta$ is such a way so that $\frac{1}{2}\delta_1=\tilde\delta$ and $\tilde\delta<< \delta_0|\abar|<<\tilde\varrho$. With this notation we can 
write $\mathcal{B} $ as
\begin{align}\label{E2.17}
\mathcal{B} &= \underbrace{\tilde\sL_q[\cone] w(\bar x)-\tilde\sL_q[\cone] \tilde{w}(\bar y)}_{=I_1}  +
\underbrace{\tilde\sL_q[\cD_1] w(\bar x)-\tilde\sL_q[\cD_1] \tilde{w}(\bar y)}_{=I_2} \nonumber
\\
&\qquad + \underbrace{\tilde\sL_q[\cD_2] w(\bar x)-\tilde\sL_q[\cD_2] \tilde{w}(\bar y)}_{=I_3}
+ \underbrace{\tilde\sL_q[B^c_{\tilde\varrho}] w(\bar x)-\tilde\sL_q[ B^c_{\tilde\varrho}] \tilde{w}(\bar y)}_{=I_4}.
\end{align}

We conclude this section by gathering suitable estimates of $I_1, I_2$ and $I_4$ from \cite{BT25,BS25}.
\begin{lem}\label{Est-I}
Let $\cone$ be the cone mentioned above and $q\in (1, \infty)$. Then there exists $L_0$, dependent on $m, m_1, \norm{u}, \gamma,$ so that the following hold.
\begin{itemize}
\item[(i)] There exists $\delta_0\in (0, \frac{1}{8})$ and a constant $C_2$, dependent on $q, s, \gamma, n$, such that
$$ I_1\geq C_2 L^{q-1}|\abar|^{\gamma(q-1)-sq}\quad \text{for all}\; \; L\geq L_0.$$

\item[(ii)] Define $\tilde\delta=\varepsilon_1 |\abar|$ with $\varepsilon_1\in (0, \frac{1}{8})$. Then for some constant $C_3$, independent of $\varepsilon_1, |\abar|$, satisfying
$$I_2\geq - C_3 \varepsilon_1^{q(1-s)} |\abar|^{\gamma(q-1)-sq}\quad \text{for all} \; L\geq L_0.$$

\item[(iii)] There exists a constant $C_4=C_{4}(\varrho_1, \varrho_2, n, s, q)$ satisfying $|I_4|\leq C_4 \norm{u}_\infty$.
\end{itemize}
\end{lem}

\begin{proof}
First, we choose $L_0$ large enough using \eqref{E2.6} so that $0<|\abar|<\frac{1}{8}$.
Then, for $q>2$, (i) follows from \cite[Lemma~2.2 and 3.1]{BT25} whereas for $q\in (1, 2]$ it can be obtained from \cite[Lemma~4.1]{BT25} (see also, \cite[Lemma~3.1]{BS25}). (ii) follows from
\cite[Lemma~3.2 and 4.2]{BT25}. (iii) is rather straightforward.
\end{proof}

%%%%%%%%%%%%%%%%%%%%%%%%%%%%%%%%%%%%%%%%%%%%%%%%%%%%%%%%%%%%%%%%
\subsection{Almost Lipschitz regularity for $q>2$}
We begin with an estimate of $I_3$ given by \eqref{E2.17}. 

\begin{lem}\label{Est-I3}
Suppose that $q> 2$ and $u\in C^{0, \upkappa}(\bar{B}_{\varrho_2})$ for some $\upkappa\in [0, 1)$. Let $\eta=\varepsilon_1|\abar|$ where $\varepsilon_1$
is given by Lemma~\ref{Est-I}(ii).
Then, we have $L_0=L_0(\varrho_1, \varrho_2, \norm{u}_\infty)>0$ such that
$$I_3\geq -\kappa \left[ \int_\eta^{\tilde\varrho} r^{\upkappa(q-2) + 1-sq} \dr  + 
|\abar|^{\frac{m-1}{m}\upkappa} \int_\eta^{\tilde\varrho} r^{\upkappa(q-2) -sq} \dr\right]$$
for all $L\geq L_0$, where the constant $\kappa$ depends on $\upkappa, q, s, n, m, m_1$ and the $C^{0, \upkappa}$ norm of $u$ in
$\bar{B}_{\varrho_2}$.
\end{lem}

\begin{proof}
Using \eqref{E2.6} we choose $L_0$ large enough so that $|\abar|<\frac{1}{2}\tilde\varrho=\frac{\varrho_2-\varrho_1}{8}$ for 
all $L\geq L_0$.

We set the notation $\triangle g(x, z)= g(x)-g(x+z)$. Using the fundamental theorem of calculus, we see that
$$I_{3}=(q-1)\int_{\cD_2}\int_0^t |\triangle u(\ybar, z) + t (\triangle u(\xbar, z)-\triangle u(\ybar, z))|^{q-2}
(\triangle u(\xbar, z)-\triangle u(\ybar, z))\, \dt \frac{\dz}{|z|^{n+sq}}.$$
Since $\Phi(\xbar+z,\ybar+z)\leq \Phi(\xbar,\ybar)$, we get
$$\triangle u(\xbar, z) - \triangle u(\ybar, z)\geq m_1 \triangle \psi(\xbar, z),$$
leading to
\begin{equation}\label{EL2.5A}
I_{1,3}\geq -(q-1)2^{q-2}\, m_1 \int_{B_{\tilde\varrho}\cap B^c_\eta} (|\triangle u(\ybar, z)| +  |\triangle u(\xbar, z)|)^{q-2} |\triangle \psi(\xbar, z)|\frac{\dz}{|z|^{n+sq}}.
\end{equation}
Note that $\xbar+z, \ybar+z\in B_{\varrho_2}$ for all $z\in B_{\tilde\varrho}$. Therefore,
$$|\triangle u(\ybar, z)| +  |\triangle u(\xbar, z)|\leq 2 [u]_{\upkappa, \varrho_2} |z|^\upkappa,$$
where $[u]_{\upkappa, \varrho_2}$ denotes the $C^{0, \upkappa}$ seminorm in $\bar{B}_{\varrho_2}$. Again, from the Taylor's expansion
of $\psi$ we also get
$$|\triangle \psi(\xbar, z)|\leq \kappa_2 (|z|^2 + |\nabla\psi(\xbar)||z|)$$
for some constant $\kappa_2$, dependent on $m$. Putting these estimates in \eqref{EL2.5A} we arrive at
\begin{align}\label{EL2.5B}
I_{3} &\geq -\kappa_3 \int_{B_{\tilde\varrho}\cap B^c_{\eta}} (|z|^{\upkappa(q-2) + 2} + |\nabla\psi(\xbar)| |z|^{\upkappa(q-2) + 1}) \frac{\dz}{|z|^{n+sq}}
\nonumber
\\
& =-\kappa_4 \int_\eta^{\tilde\varrho} (r^{\upkappa(q-2) + 2} + |\nabla\psi(\xbar)| r^{\upkappa(q-2) + 1}) r^{-1-sq} \dr
\end{align}
for some constants $\kappa_3, \kappa_4$, dependent on $[u]_{\upkappa, \varrho_2}, m, m_1, n$. From \eqref{E2.5} we have
$$\psi(\xbar)\leq \frac{1}{m_1} (u(\xbar)-u(\ybar))\leq \frac{1}{m_1} [u]_{\upkappa, \varrho_2} |\abar|^\upkappa.$$
Thus
$$|\nabla\psi(\xbar)|\leq 2m (\psi(\xbar))^{\frac{m-1}{m}}\leq \kappa_5 |\abar|^{\frac{m-1}{m}\upkappa}$$
for some constant $\kappa_5$. Hence from \eqref{EL2.5B} we obtain
\begin{equation*}
I_{3} \geq -\kappa_6 \left[ \int_\eta^{\tilde\varrho} r^{\upkappa(q-2) + 1-sq} \dr  + 
|\abar|^{\frac{m-1}{m}\upkappa} \int_\eta^{\tilde\varrho} r^{\upkappa(q-2) -sq} \dr\right].
\end{equation*}
Hence the proof.
\end{proof}

Now we are ready to prove almost Lipschitz regularity for $q>2$.
\begin{thm}\label{T2.6}
Let $q>2$ and $u\in C(\Rn)\cap L^\infty(\Rn)$ satisfy \eqref{E2.3}. Then for any ball $B\Subset \Omega$ we have $u\in C^{0, \beta}(\bar{B})$ for any $\beta\in (0, 1)$. Moreover,
the $C^{0, \beta}$ norm of $u$ in $B$ depends only on $C, \data, \beta$ and $\dist(B, \partial\Omega)$.
\end{thm}

\begin{proof}
As described before, for the economy of notations, we assume $B=B_1\Subset B_2\Subset \Omega$. Also, set $1\leq \varrho_1<\varrho_2\leq 2$ and $\Phi$ as in \eqref{E2.4}. We show that there
exists $L_0$, depending on  $C, \data$\, and $\dist(B, \partial\Omega)$, such that $\Phi\leq 0$ in $\Rn\times\Rn$ for all $L\geq L_0$. This, in particular, would imply that
$$ |u(x)-u(y)|\leq L|x-y|^{\gamma}\quad \text{for all}\; \; x, y\in \bar{B}_{\varrho_1},$$
completing the proof.

The above result is proved using the method of contradiction. So assume that \eqref{E2.5} holds, and arrive at \eqref{E2.16}. Moreover, using Lemma~\ref{Est-I} we can choose $\varepsilon_1$ small
enough in comparison to $C_2$ so that 
$$ I_1+I_2+I_4\geq \frac{C_2}{2}L^{q-1}|\abar|^{\gamma(q-1)-sq}-C_4\norm{u}_\infty$$
for all $L\geq L_0$, where $L_0$ depends on $m, m_1, \gamma$ and $\norm{u}_\infty$. Now fix this $\varepsilon_1$.
Plugging it in \eqref{E2.16} we obtain
\begin{equation}\label{ET2.6A}
\tilde{C}_1 L^{p-1} |\abar|^{\gamma(p-1)-p} + \frac{C_2}{2}L^{q-1}|\abar|^{\gamma(q-1)-sq}+ I_3
\leq 2C + C_4\norm{u}_\infty,
\end{equation}
for all $L\geq L_0$, where $\tilde{C}_1=\gamma^{p-1} (1-\gamma) C_1$.

The proof uses an iteration procedure. Let us define $\gamma_s=\frac{sq}{q-1}$. Suppose that $u\in C^{0, \upkappa}(\bar{B}_{\varrho_2})$ for some $\upkappa\in [0, \gamma_s\wedge 1)$.
The case $\upkappa=0$ corresponds to the situation when $u$ is merely continuous. We set $\gamma\in (0, \min\{\gamma_s, \upkappa +\frac{1}{q-1} \})$ and $m\geq 3$ large enough
so that 
$$\upkappa(q-2) +\frac{m-1}{m}\upkappa + 1> \gamma(q-1), \quad \text{and}\quad \frac{m-1}{m}\upkappa<\gamma_s.$$
Due to this choice we have 
$$ \upkappa(q-2)-sq> \gamma(q-1)-sq-1 -\frac{m-1}{m}\upkappa.$$
Since $\gamma<\frac{sq}{q-1}\Rightarrow \gamma(q-1)-sq<0$, we obtain
\begin{align*}
|\abar|^{\frac{m-1}{m}\upkappa}\int_{\varepsilon_1 |\abar|}^{\tilde\varrho} r^{\upkappa(q-2) -sq} \dr &\leq 
|\abar|^{\frac{m-1}{m}\upkappa}\int_{\varepsilon_1 |\abar|}^{1} r^{\gamma(q-1)-sq-1 -\frac{m}{m-1}\upkappa} \dr
\\
&=\frac{|\abar|^{\gamma(q-1)-sq}}{sq+\frac{m}{m-1}\upkappa-\gamma(q-1)} (\varepsilon_1)^{\gamma(q-1)-sq -\frac{m-1}{m}\upkappa}.
\end{align*}
On the other hand,
\begin{align*}
\int_{\varepsilon_1 |\abar|}^{\tilde\varrho} r^{\upkappa(q-2)+1 -sq} \dr 
&\leq \int_{\varepsilon_1 |\abar|}^{1} r^{\gamma(q-1)-sq -\frac{m-1}{m}\upkappa} \dr
\\
&\leq\left\{\begin{array}{lll}
\frac{1}{\gamma(q-1)-sq +1 -\frac{m-1}{m}\upkappa} & \text{if}\; \gamma(q-1)-sq -\frac{m-1}{m}\upkappa>-1,
\\
-\log(\varepsilon_1|\abar|) & \text{if}\; \gamma(q-1)-sq -\frac{m-1}{m}\upkappa=-1,
\\
\frac{(\varepsilon_1|\abar|)^{\gamma(q-1)-sq+1-\frac{m-1}{m}\upkappa}}{sq + \frac{m-1}{m}\upkappa-1-\gamma(q-1)}
& \text{if}\; \gamma(q-1)-sq -\frac{m-1}{m}\upkappa<-1.
\end{array}
\right.
\end{align*}
Since $1-\frac{m-1}{m}\upkappa>0$ and $|\log(\varepsilon_1 |\abar|)|\leq \kappa (\varepsilon_1|\abar|)^{\gamma(q-1)-sq}$, for some
constant $\kappa$, from the above estimates and Lemma~\ref{Est-I3} we can find a constant $C_{\varepsilon_1}$, dependent on
$\varepsilon_1, q, s, m$, so that $I_3\geq -C_{\varepsilon_1} (|\abar|^{\gamma(q-1)-qs}+1)$ for all $L\geq L_0=L_0(\varrho_1, \varrho_2)$.
Thus, from \eqref{ET2.6A}, we find $L_0$, dependent on the 
$\data, \gamma, \upkappa$, satisfying
$$ \tilde{C}_1 L^{p-1} |\abar|^{\gamma(p-1)-p} + \frac{C_2}{2}L^{q-1}|\abar|^{\gamma(q-1)-sq}- C_{\varepsilon_1}|\abar|^{\gamma(q-1)-sq}
\leq 2C + C_4\norm{u}_\infty + C_{\varepsilon_1}$$
for all $L\geq L_0$. Since $|\abar|\to 0$ as $L\to\infty$, by \eqref{E2.6}, $\gamma(q-1)-sq<0$ and $\varepsilon_1$ does not depend on $L$,
the above inequality can not hold for large enough $L$, leading to a contradiction. Hence for any $L\geq L_0$ for which the above inequality
fails to hold, we must have $\Phi\leq 0$ in $\Rn\times\Rn$. In particular, $u\in C^{0, \gamma}(\bar{B}_{\varrho_1})$.

Now for any ball $D$ satisfying $B_1\Subset D \Subset B_2$ and $\gamma<\gamma_s\wedge 1$, we can apply the above argument over a strictly decreasing sequence of finitely many balls to conclude that $u\in C^{0, \gamma}(\bar{D})$. Moreover, the $C^{0, \gamma}$ norm on $B$ depends only on the 
$\data, \gamma$ and $\dist(D,\partial\Omega)$. Therefore, if $\gamma_s\geq 1$, we have our proof letting $\beta=\gamma$.

Next, we suppose $\gamma_s<1$. From the first part of the proof we have $u\in C^{0, \upkappa}(\bar{B}_{\varrho_2})$ for any 
$\upkappa<\gamma_s$. Choose $\upkappa$ close to $\gamma_s$ so that $1+\upkappa(q-2)-sq>0$. 
Take $\beta\in [\gamma_s, 1)$. Again, we compute a lower bound for $I_3$ in \eqref{ET2.6A}. Since
\begin{align*}
\int_{\varepsilon_1 |\abar|}^{\tilde\varrho} r^{\upkappa(q-2) -sq} \dr &\leq \frac{1}{\upkappa(q-2)+1-sq},
\\
\int_{\varepsilon_1 |\abar|}^{\tilde\varrho} r^{\upkappa(q-2)+1 -sq} \dr &\leq \frac{1}{\upkappa(q-2)+2-sq},
\end{align*}
from Lemma~\ref{Est-I3} and \eqref{ET2.6A} we have
$$\tilde{C}_1 L^{p-1} |\abar|^{\beta(p-1)-p} + \frac{C_2}{2}L^{q-1}|\abar|^{\beta(q-1)-sq}+ I_3
\leq 2C + C_4\norm{u}_\infty+ C_5,$$
for some constant $C_5$ and $L\geq L_0$. Again, since $\beta(p-1)-p<0$, the above inequality can not hold for large enough $L$, leading to 
a contradiction. Arguing as before we get $u\in C^{0, \beta}(\bar{B}_{\varrho_1})$.
This completes the proof.
\end{proof}

\subsection{Almost Lipschitz regularity for $q\in (1,2]$}

From \cite[Lemma~4.3]{BT25} and the argument of Lemma~\ref{Est-I3} we have
\begin{lem}\label{Est-I3-q2}
Suppose that $q\in (1,2]$ and $u\in C^{0, \upkappa}(\bar{B}_{\varrho_2})$ for some $\upkappa\in [0, 1)$. Let $\eta=\varepsilon_1|\abar|$.
Then, we have $L_0=L_0(\varrho_1, \varrho_2, \norm{u}_\infty)>0$ such that
$$I_3\geq -\kappa \left[ \int_\eta^{\tilde\varrho} r^{2(q-1)-sq-1} \dr  + 
|\abar|^{\frac{(m-1)(q-1)}{m}\upkappa} \int_\eta^{\tilde\varrho} r^{q-sq-2} \dr\right]$$
for all $L\geq L_0$, where the constant $\kappa$ depends on $\upkappa, q, s, n, m, m_1$ and the $C^{0, \upkappa}$ norm of $u$ in
$\bar{B}_{\varrho_2}$.
\end{lem}

Now we prove almost Lipschitz regularity for $q\in (1,2]$.
\begin{thm}\label{T2.8}
Let $q\in (1,2]$ and $u\in C(\Rn)\cap L^\infty(\Rn)$ satisfy \eqref{E2.3}. Then for any ball $B\Subset \Omega$ we have $u\in C^{0, \beta}(\bar{B})$ for any $\beta\in (0, 1)$. Moreover,
the $C^{0, \beta}$ norm of $u$ in $B$ depends only on $C, \data, \beta$ and $\dist(B, \partial\Omega)$.
\end{thm}

\begin{proof}
We follow the proof of Theorem~\ref{T2.6} and therefore, we keep the same notation as in Theorem~\ref{T2.6}. As before, we 
assume $B=B_1\Subset B_2\Subset \Omega$. $\varrho_1, \varrho_2$ and $\Phi$ are the same as in Theorem~\ref{T2.6}. Note that
\eqref{ET2.6A} holds in this case as well for $L\geq L_0$, where $L_0$ depends on $m, m_1, \gamma$ and $\norm{u}_\infty$.
As before, we also denote by $\gamma_s=\frac{sq}{q-1}$.

Since $2(q-1)-sq-1> q-sq-2$, and for $\gamma<\gamma_s\wedge 1$, $q-sq-2>\gamma(q-1)-sq-1$, from Lemma~\ref{Est-I3-q2} we get
\begin{align*}
I_3\geq -2\kappa \int_{\varepsilon_1|\abar|}^{1} r^{\gamma(q-1) -sq-1}\dr
= -\frac{2\kappa}{\gamma(q-1)-sq} (\varepsilon_1|\abar|)^{\gamma(q-1)-sq}.
\end{align*}
Putting it in \eqref{ET2.6A} we arrive at
\begin{equation*}
\tilde{C}_1 L^{p-1} |\abar|^{\gamma(p-1)-p} + \frac{C_2}{2}L^{q-1}|\abar|^{\gamma(q-1)-sq}
-\frac{2\kappa}{\gamma(q-1)-sq} (\varepsilon_1|\abar|)^{\gamma(q-1)-sq}
\leq 2C + C_4\norm{u}_\infty,
\end{equation*}
for all $L\geq L_0$. This clearly can not hold for all large $L$ and we conclude that $u\in C^{0, \gamma}(\bar{B}_{\varrho_1})$ arguing as in Theorem~\ref{T2.6}.

If $\gamma_s\geq 1$, we are done with the proof by letting $\gamma=\beta$ above. So we assume that $\gamma_s<1$. From the above argument, we have $u\in C^{0, \upkappa}(\bar{B}_{\varrho_2})$ for any $\upkappa<\gamma_s$. Furthermore, $sq<q-1$ implies $2(q-1)-sq> q-1-sq>0$, giving us
$$\int_\eta^{\tilde\varrho} r^{2(q-1)-sq-1} \dr\leq \frac{1}{2(q-1)-sq},
\quad \text{and}\quad \int_\eta^{\tilde\varrho} r^{q-sq-2} \dr\leq \frac{1}{q-1-sq}.$$
From Lemma~\ref{Est-I3-q2}, this gives us $I_3\geq -\kappa$.
Now the proof can be completed along the lines of Theorem~\ref{T2.6}.
\end{proof}

Now we conclude this section with a proof of Theorem~\ref{T-main1}.
\begin{proof}
The proof follows from Theorems~\ref{T2.6} and~\ref{T2.8}, and a standard covering argument.
\end{proof}

\section{Boundary regularity of $p$-harmonic functions}\label{S-bdry}
In this section, we study up to the boundary regularity of a $p$-harmonic function with a H\"{o}lder continuous boundary data. This result will be
crucial for us to construct a valid test function in the next section, leading to the proof of $C^{1, \alpha}$ estimate.

In this section, we assume  $\Omega$ to be a bounded $C^2$ domain. Our main result of this section is as follows.
\begin{thm}\label{T-3.1}
Let $g\in W^{1,p}(\Omega)\cap C^{0, \beta}(\bar\Omega)$ for some $\beta\in (0, 1)$.
Let $v\in g+W^{1,p}_0(\Omega)$ be the unique solution to
$$w\mapsto \min_{w\in g+W^{1,p}_0(\Omega)} \frac{1}{p}\int_{\Omega} |\grad w|^p \dx.$$
Then $v\in C^{0, \beta}(\bar\Omega)$ and for some constant $\tilde{C}$, dependent on $n, p, \Omega$, it holds
$$ \norm{v}_{C^{0, \beta}(\bar\Omega)}\leq \tilde{C} \, \norm{g}_{C^{0, \beta}(\bar\Omega)}.$$
\end{thm}

 Before proceeding with the proof, we remark that the above result cannot, in general, be extended to the case $\beta=1$. See, for instance, \cite{HS99}, where the case  $p=2$ is studied, as well as the result of Hardy and Littlewood (Theorem 1 therein), which provides a sharp blow-up estimate for the gradient near the boundary. Assuming $C^{1, \alpha}$ boundary data, global $C^{1, \alpha}$ estimates for solutions of degenerate elliptic problems are presented, for instance,  in~\cite{BD14, AS23}. %We have preferred to give a direct proof in the case for $C^{0,\beta}$ boundary data instead of approximation methods using the mentioned $C^{1, \alpha}$ estimates and stability. 

We define $\D(x)=\dist (x, \partial\Omega)$. It is well-known that for some $\rho>0$, $\D$ is $C^2$ in $\bar\Omega_{\rho}$ where
$\Omega_{\rho}=\{x\in\Omega \: :\; \D(x)<\rho\}$, see \cite[Theorem~5.4.3]{DZ11}.
 We extend $\D$ as a $C^2$ function in
$\bar\Omega$.
 For $y \in \partial \Omega$, consider the function $\uppsi = \uppsi_y : \bar \Omega \to \R$ given by
\begin{equation*}
	\uppsi(x) = g(y) + M_1 \beta^{-1}|x - y|^\beta + M_2 \beta^{-1} (\D(x))^\beta,
\end{equation*}
where $M_1, M_2$ are positive constants to be determined later. We start with the following estimate on the barrier function.
\begin{lem}\label{L-Bdry}
For each $\beta \in (0,1)$ and $M_ 1 > 0$, there exists $M_\circ > M_1$, $\rho_0 \in (0,1)$ and $\kappa_p > 0$ such that, for all $y \in \partial \Omega$, we have 
$$
\Delta_p \uppsi(x) \leq \kappa_p  (\beta - 1)  M_2^{p-1} \D(x)^{\beta(p-1) - p},
$$
for all $x \in \Omega$ with $\D(x) \leq \rho_0$ and $M_2\geq M_\circ$.
\end{lem}
	
\begin{proof}
First of all, we note that
\begin{equation}\label{E3.1}
 \grad \uppsi(x)= M_1 |x-y|^{\beta-2}(x-y)+ M_2(\D(x))^{\beta-1}\grad \D(x).
 \end{equation}
 
Since $|\grad \D(x)| = 1$ for $x$ close to $\partial\Omega$, and $\D(x)\leq |x-y|$, for $M_2\geq 2M_1$, we can choose $\rho_0$ small enough
so that $|\grad\uppsi|>0$ for $\D(x)\leq\rho_0$. Next  we compute 
\begin{equation}\label{pLap}
\Delta_p \uppsi(x) = (p - 2)|\grad\uppsi(x)|^{p-4} \langle \grad\uppsi(x), D^2 \uppsi(x) \grad\uppsi(x)\rangle + |\grad\uppsi(x)|^{p-2} \Delta \uppsi(x).	
\end{equation}
An easy calculation reveals that
\begin{equation}\label{D2psi}
D^2 \uppsi(x) = A_1 + A_2 + A_3 + A_4,
\end{equation}
where
\begin{align*}
A_1 = & M_1 (\beta-2)|x-y|^{\beta-2} \widehat{(x - y)} \otimes \widehat{(x - y)},
\\
A_2 = & M_1 |x - y|^{\beta - 2} I_d ,
\\
A_3 = & M_2 (\beta - 1) (\D(x))^{\beta - 2} \grad \D(x) \otimes \grad \D(x) ,
\\
A_4 = & M_2 (\D(x))^{\beta - 1} D^2 \D(x).
\end{align*}

Here $\hat{a}$ denotes the unit vector along $a$, that is, $\hat{a}=a/|a|$.	
% Noticing that $A_1$ is negative semi-definite,  $d(x) \leq |x - y|$ and  $|\grad\D(x)| \approx 1$ for $x$ close to the boundary, we have 
%	\begin{align*}
%	\Delta \uppsi(x) & =  \mathrm{Tr}(\sum_i A_i)
%	 \\
%	&\leq  M_1 n |x - y|^{\beta - 2} + M_2 (\beta - 1) (\D(x))^{\beta - 2} |\grad\D(x)|^2 + \norm{D^2\D(x)}_\infty M_2 (\D(x))^{\beta - 1}
%	\\
%	& \leq  (M_1 n + M_2 (\beta - 1)|\grad\D(x)|^2 )(\D(x))^{\beta - 2}  + \norm{D^2\D(x)}_\infty M_2 (\D(x))^{\beta - 1}.
%	\end{align*}
% Thus, we can choose  $M_2$ large enough, depending on  $M_1, n$ and $\beta$, and $\rho_0$ small enough, depending on
% $\norm{D^2\D(x)}_\infty$, so that
%\begin{equation}\label{Lap}
%\Delta \uppsi(x)	\leq \frac{M_2}{2} (\beta - 1)(\D(x))^{\beta - 2}
%\end{equation}
%for all $\D(x)\leq \rho_0$. 
For all calculations below we set $\rho_0$ small enough so that $\frac{1}{2}\leq |\grad \D(x)|\leq \frac{3}{2}$ for $\D(x)\leq \rho_0$.	
Now we deal with the first term in \eqref{pLap}. It is easy to see that
\begin{align}\label{E3.4}
\langle \grad\uppsi(x), D^2 \uppsi(x) \grad\uppsi(x) \rangle = & M_1 (\beta-2) |x-y|^{\beta-2} \langle \grad\uppsi(x), \widehat{(x - y)} \rangle^2
+ M_1 |x - y|^{\beta - 2} |\grad\uppsi(x)|^2 \nonumber
\\
& + M_2 (\beta - 1) (\D(x))^{\beta - 2} \langle \grad\uppsi(x), \grad \D(x) \rangle^2\nonumber
 \\
&\quad + M_2 (\D(x))^{\beta -1} \langle \grad\uppsi(x), D^2 \D(x) \grad\uppsi(x) \rangle.
\end{align}
Since the first term on the rhs of \eqref{E3.4} is negative and that $|x - y|\geq \D(x)$, we have 
\begin{align*}
\langle \grad\uppsi(x), D^2 \uppsi(x) \grad\uppsi(x) \rangle \leq & (M_1 + M_2 \norm{D^2\D}_\infty \D(x)) (\D(x))^{\beta - 2} |\grad\uppsi(x)|^2
 \\
&  + M_2 (\beta - 1) (\D(x))^{\beta - 2} \langle \grad\uppsi(x), \grad \D(x) \rangle^2.
\end{align*}
For the last term above, we use \eqref{E3.1} and write
\begin{align*}
 \langle \grad\uppsi(x), \grad \D(x) \rangle^2 
& =  M_1^2 |x - y|^{2(\beta - 1)} \langle \widehat{(x - y)}, \grad \D(x)\rangle^2 \\
&  + 2M_1 M_2 |x - y|^{\beta - 1} (\D(x))^{\beta - 1} \langle \widehat{(x - y)}, \grad \D(x) \rangle |\grad \D(x)|^2  \\
& + M_2^2 (\D(x))^{2(\beta - 1)} |\grad \D(x)|^4.
\end{align*}
Using the fact $|x - y|\geq \D(x)$, this leads to, for $\D(x)\leq \rho_0$,
\begin{align*}
\langle \grad\uppsi(x), \grad \D(x) \rangle^2 & \geq M_2 (\D(x))^{2(\beta - 1)} |\grad \D(x)|^3 
\Bigl( -2M_1 + M_2 |\grad \D(x)|\Bigr )
\\
& \geq M_2 (\D(x))^{2(\beta - 1)} |\grad \D(x)|^3 
\Bigl( -2M_1 + \frac{M_2}{2} \Bigr )
 \geq \frac{M_2^2}{32} (\D(x))^{2(\beta - 1)}
\end{align*}
for any $M_2\geq 8 M_1$. Thus, gathering the above estimates, we arrive at
\begin{align*}
\langle \grad\uppsi(x), D^2 \uppsi(x) \grad\uppsi(x) \rangle & \leq  (M_1 + M_2 \norm{D^2\D}_\infty \D(x)) (\D(x))^{\beta - 2} |\grad\uppsi(x)|^2
 \\
&\quad  + \frac{M_2^3}{32}(\beta - 1) (\D(x))^{\beta - 2 + 2(\beta - 1)}.
\end{align*}
Also, by \eqref{E3.1}
\begin{align*}
|\grad\uppsi(x)|^2 = M_1^2 |x - y|^{2(\beta - 1)} + M_2^2 \D(x)^{2(\beta -1)}|\grad\D(x)|^2 + 2M_1M_2 |x - y|^{\beta - 1} (\D(x))^{\beta -1}
 \langle \widehat{(x - y)}, \grad\D(x) \rangle,
\end{align*}
from which, letting $M_1 \leq M_2/16$, we conclude that
\begin{equation}\label{estDpsi}
	\frac{M_2}{4} (\D(x))^{\beta - 1} \leq |\grad\uppsi(x)| \leq 2 M_2 (\D(x))^{\beta - 1}. 				
\end{equation}
Using the upper bound in \eqref{estDpsi}, we obtain
\begin{align*}
\langle \grad\uppsi(x), D^2 \uppsi(x) \grad\uppsi(x) \rangle 
\leq M_2^2 (\D(x))^{3\beta - 4} \Bigl( 4M_1 + 4M_2 \norm{D^2\D}_\infty \D(x) + (\beta - 1)\frac{M_2}{32}\Bigr),
\end{align*}
and therefore, taking $M_2$ large depending on $\beta$ and $M_1$, and then letting $\rho_0$ small, if required, we arrive at
\begin{align*}%\label{inftyLap}
\langle \grad\uppsi(x), D^2 \uppsi(x) \grad\uppsi(x) \rangle \leq (\beta - 1)\frac{M_2^3}{64}  (\D(x))^{3\beta - 4}.
\end{align*}
Combining it with \eqref{estDpsi} we get
\begin{align}\label{inftyLap}
\langle \widehat{\grad\uppsi(x)}, D^2 \uppsi(x) \widehat{\grad\uppsi(x)} \rangle \leq \kappa_1 (\beta - 1)M_2  (\D(x))^{\beta - 2}.
\end{align}
for some constant $\kappa_1$.

%Gathering~\eqref{Lap},~\eqref{inftyLap}, and the bounds of $|\grad\uppsi(x)|$ in \eqref{estDpsi} in~\eqref{pLap}, we conclude that
%\begin{align*}
%\Delta_p \uppsi(x) & \leq  \kappa_1 (\beta - 1) M_2^{p-1} (\D(x))^{\beta (p-1)-p}   + \kappa_2 M_2^{p-1}  (\beta - 1) (\D(x))^{(\beta )(p-1)-p}
% \\
%& \leq  (\kappa_1+\kappa_2) (\beta - 1)  M_2^{p-1}  (\D(x))^{\beta (p-1)-p} 
%\end{align*}
%for some constant $\kappa_1, \kappa_2$. This gives us the required estimate.

%\noindent{\bf Case 2.} We now suppose that $p\in (1, 2)$. In this case,

We write~\eqref{pLap} as
\begin{align*}
\Delta_p \uppsi(x) = |\grad\uppsi(x)|^{p-2} \Big{\{}(p - 2) \langle \widehat {\grad\uppsi(x)}, D^2 \uppsi(x) \widehat{\grad\uppsi(x)} \rangle + \Delta \uppsi(x) \Big{\}}.
\end{align*}
%Since $|x - y| \geq d(x)$, we see that
%\begin{align*}
%\langle Dd(x), \widehat {D\psi(x)} \rangle = 1 + O(M_1 M_2^{-1}),
%\end{align*}
%from which, as in Lemma~\ref{L2.2},
 We consider a set of  orthonormal basis of $\R^n$ given by $\{ v_1, ..., v_n\}$ with $v_1= \widehat{D\uppsi(x)}$.
Then
		\begin{align}\label{pLap2}
		\Delta_p \uppsi(x) = |D\uppsi(x)|^{p-2} \Big{\{}(p - 1) \langle \widehat {D\uppsi(x)}, D^2 \uppsi(x) \widehat{D\uppsi(x)} \rangle + \sum_{i=2}^n \langle v_i, D^2 \uppsi(x) v_i \rangle \Big{\}},
		\end{align}
from which, since $p-1 > 0$, using \eqref{inftyLap} we conclude that
\begin{equation}\label{pLap21}
(p - 1) \langle \widehat {D\uppsi(x)}, D^2 \uppsi(x) \widehat{D\uppsi(x)} \rangle \leq (p-1)\kappa_1 (\beta - 1)M_2  (\D(x))^{\beta - 2}. 
\end{equation}
On the other hand, using \eqref{D2psi} and the non-positivity of $A_1$ and $A_3$, we obtain	
\begin{align*}
\sum_{i=2}^n \langle v_i, D^2 \uppsi(x) v_i \rangle & \leq  M_1(n-1) |x - y|^{\beta - 2} + M_2 (n-1) \norm{D^2\D}_\infty (\D(x))^{\beta-2}
\\
&\leq (M_1(n-1) + M_2 (n-1) \norm{D^2\D}_\infty \D(x)) (\D(x))^{\beta-2}.
\end{align*}
Hence, combining it with \eqref{pLap21}, choosing $M_2$ large and $\rho_0$ small, and using \eqref{estDpsi} in \eqref{pLap2} we conclude that
$$ \Delta_p \uppsi(x) \leq \kappa_2 (\beta - 1)  M_2^{p-1}  (\D(x))^{\beta (p-1)-p}$$
for some constant $\kappa_2$. Hence the proof.
\end{proof}

Coming back to Theorem~\ref{T-3.1}, we see that 
\begin{equation}\label{E3.9}
- \Delta_p v =0\quad \text{in}\; \Omega, \quad \text{and}\quad v=g\quad \text{on}\; \partial\Omega.
\end{equation}
Since $g$ is continuous, from the boundary regularity results of Maz'ya-Wiener it is known that $u\in C(\bar\Omega)$ and $v=g$ on $\partial\Omega$
(\cite{Lindq,Maz}, \cite[Corollary~4.18]{MZ97}). Now we complete the proof of Theorem~\ref{T-3.1}.

%	\begin{prop}
%	Let $u \in C(\Omega)$ solve the Dirichclet problem $\Delta_p u = 0$ in $\Omega$ with $u = g$ on $\partial \Omega$ with $g$ $\beta$ H\"older continuous with $[g]_{C^\beta(\partial \Omega)} = L_g < +\infty$. Then, $u$ is in $C^{0,1}_{loc}(\Omega) \cap C^\beta(\bar \Omega)$. {\color{red} Furthermore,
%	$$\norm{u}_{C^{0, \beta}}\leq C \norm{g}_{C^{0, \beta}},$$
%	where $C$ depends on $p, n, \Omega$.
%	}
%	\end{prop}
	
\begin{proof}[Proof of Theorem~\ref{T-3.1}]
	Firstly, we can normalize $g$ so that $\norm{g}_{C^{0, \beta}(\bar\Omega)}=1$. From the maximum principle, it then follows that 
	$$\max_{\bar\Omega}|v|\leq \max_{\bar\Omega}|g|\leq 1.$$
	Set $M_1= \frac{2}{\beta}$. Define $\Omega_{\rho_0}=\{x\in\Omega\; :\; \D(x)<\rho_0\}$, where $\rho_0$ is given by Lemma~\ref{L-Bdry}. 
	We can choose $M_2$ large enough so that, for any $y\in\partial\Omega$, we have
	$$v\leq \uppsi_y(x)\quad \text{for}\; x\in \Omega^c_{\rho_0}\cap\bar\Omega.$$
	and $\grad\uppsi_y(x)\neq 0$ in $\Omega_{\rho_0}$. Since $v$ is also a viscosity solution to \eqref{E3.9} in the sense of Proposition~\ref{P1.2} (see also, \cite{JLM}), using maximum principle
	 it is easily seen that $v\leq \psi_y$ in $\bar\Omega$ for all $y$. For, otherwise, we can find a 
	positive $\theta$ and a point $z\in \Omega_{\rho_0}$ such that $\psi_y(z)+\theta=v(z)$ and $\psi_y+\theta\geq v$ in $\Omega_{\rho_0}$. From the definition of viscosity solution
	(analogous to Proposition~\ref{P1.2}) it follows that $\Delta_p(\uppsi_y(z)+\theta)=\Delta_p\uppsi_y(z)\geq 0$, which is a contradiction to Lemma~\ref{L-Bdry}.
	
	Hence, letting 
	$$
	\Psi^+(x) = \inf_{y \in \partial \Omega} \uppsi_y(x), \quad x \in \bar \Omega,
	$$
	 we have $v\leq \Psi^+$ in $\bar\Omega$. Similarly, replacing $v$ by $-v$, we see that for
	$$
	\Psi^-(x) := \sup_{y \in \partial \Omega} g(y) - M_1 \beta^{-1}|x - y|^\beta - M_2 \beta^{-1} (\D(x))^\beta, \quad x \in \bar \Omega,
	$$
with the same choice of $M_1, M_2$ as above, we have $\Psi^-(x) \leq v(x)$ in $\bar\Omega$.
  Now, we perform Ishii-Lions method \cite{IL90}, by considering
	$$
	\max_{x, y\in\bar\Omega}\Phi_1(x, y)=\max_{x, y \in \bar \Omega} \{ v(x) - v(y) - L|x - y|^\beta \},
	$$
	for $L$ suitably large. It is evident that if the above maximum is non-positive we have the result. So, we suppose, on the contrary, that
$\max_{x, y\in\bar\Omega}\Phi_1(x, y)>0$ and $\tilde{x}, \tilde{y}\in\bar\Omega$ satisfies
$$\Phi_1(\tilde{x},\tilde{y})= \max_{x, y\in\bar\Omega}\Phi_1(x, y)>0.$$
Clearly, $\tilde{x}\neq \tilde{y}$.
Now, if we let $L \geq 2M_2\beta^{-1} + 1$ (keep in mind that $M_2 > M_1>1$ by Lemma~\ref{L-Bdry}), then $(\tilde{x},\tilde{y})\notin \partial\Omega\times\partial\Omega$. Otherwise, 
$$
0 < v(\tilde x)- v(\tilde y) - L|\tilde x - \tilde y|^\beta = g(\tilde x) - g(\tilde y) - L|\bar x - \bar y|^\beta \leq (1 - L) |\tilde x - \tilde y|^\beta < 0,
$$
which is not possible. Again, if $\tilde{y}\in \partial\Omega$, then 
\begin{align*}
0 & <  v(\tilde x) - v(\tilde y) - L|\tilde x - \tilde y|^\beta \\
 & =  v(\tilde x) - g(\tilde y) - L|\tilde x - \tilde y|^\beta \\
& \leq  \Psi^+(\tilde x) - g(\tilde y) - L|\tilde x - \tilde y|^\beta \\
& \leq  g(\tilde y) + M_1\beta^{-1} |\tilde x - \tilde y|^\beta + M_2{\beta}^{-1} (\D(\tilde x))^\beta - g(\tilde y) - L|\tilde x - \tilde y|^\beta,
\\
&\leq (M_1\beta^{-1}+M_2\beta^{-1}-L)|\tilde x - \tilde y|^\beta<0,
\end{align*}
which is also not possible. Similarly, $\tilde{x}\notin\partial\Omega$. Therefore, we must have $(\tilde{x},\tilde{y})\in\Omega\times\Omega$. Define
$\tilde\varphi(t)=t^\beta$. Now we can work with standard Ishii-Lions technique. Note that Lemma~\ref{LJI} is also applicable in this setting and \eqref{E2.16} turns out to be
$$ -C_1 L^{p-1} (\varphi'(|\tilde{a}|))^{p-2}\varphi^{\prime\prime}(|\tilde{a}|)\leq 0,$$
for some constant $C_1$ and $L$ large, where $\tilde{a}=\tilde{x}-\tilde{y}$. Since the lhs of the above display is positive, we arrive at a contradiction. This proves that $\Phi_1\leq 0$
in $\Omega\times\Omega$, which concludes our proof.
\end{proof}

\smallskip

In what follows, for a generic function $w: B \to \R$ we denote
\begin{equation}\label{def-mean}
[w]_{C^{0, \beta}(B_1)} =\sup_{x\neq y: x, y\in B_1}\frac{|w(x)-w(y)|}{|x-y|^\beta}\quad\text{and}\quad (w)_{B}=\fint_{B} w(z)\dz.
\end{equation}

We need the following corollary of Theorem~\ref{T-3.1}.
\begin{cor}\label{Cor3.3}
Let $g\in W^{1, p}(B_{r_\circ}(x_0))\cap C^{0, \beta}(\bar{B}_{r_\circ}(x_0))$ for some $\beta\in (0, 1)$. For $r\in (0, r_\circ]$, let 
$v_r\in g+W^{1,p}_0(B_r(x_0))$ be the unique solution to
$$w\mapsto \min_{w\in g+W^{1,p}_0(B_r(x_0))} \frac{1}{p}\int_{B_r(x_0)} |\grad w|^p \dx.$$
Then there exists a constant $C_{r_\circ}$, independent of $r$ and $g$, satisfying
$$ \sup_{x\neq y: x, y\in B_r}\frac{|v_r(x)-v_r(y)|}{|x-y|^\beta}\leq C_{r_\circ} \sup_{x\neq y: x, y\in B_r}\frac{|g(x)-g(y)|}{|x-y|^\beta}.$$
\end{cor}

\begin{proof}
Without any loss of generality, we assume that $x_0=0$ and $r_\circ=1$. For any $r\in (0, 1]$, we see that $\tilde{v}= r^{\frac{n}{p}-1}v_r(rx)$ is the minimizer of 
$$w\mapsto \min_{w\in \tilde{g}+W^{1,p}_0(B_1)} \frac{1}{p}\int_{B_1} |\grad w|^p \dx.$$
where $\tilde{g}(x)= r^{\frac{n}{p}-1}g(rx)$. Applying Theorem~\ref{T-3.1}, we find a constant $C_3$, independent of $\tilde{g}$ and $\tilde{v}$, satisfying
$$ \norm{\tilde{v}}_{C^{0, \beta}(\bar{B}_1)}\leq C_3 \norm{\tilde{g}}_{C^{0, \beta}(\bar{B}_1)}.$$
In particular, we have
\begin{equation}\label{EC3.3A}
 \sup_{x\neq y: x, y\in B_1}\frac{|\tilde v(x)-\tilde v(y)|}{|x-y|^\beta}\leq C_3 \left[\max_{\bar{B}_1}|\tilde g| + \sup_{x\neq y: x, y\in B_1}\frac{|\tilde g(x)-\tilde g(y)|}{|x-y|^\beta}\right].
\end{equation}

From the uniqueness of the minimizer we note that $\tilde{v}-(\tilde{g})_{B_1}$ minimizes the same functional with boundary data
$\tilde{g}-(\tilde{g})_{B_1}$. Thus, \eqref{EC3.3A} holds, if we replace $\tilde{v}$ and $\tilde{g}$ by $\tilde{v}-(\tilde{g})_{B_1}$ and $\tilde{g}-(\tilde{g})_{B_1}$, respectively.
Again, for any $x\in B_1$, we have
$$| \tilde{g}(x)-(\tilde{g})_{B_1}|\leq \fint_{B_1} |\tilde{g}(x)-\tilde{g}(z)|\dz\leq [\tilde{g}]_{C^{0,\beta}(B_1)}\Rightarrow \max_{\bar{B}_1}| \tilde{g}(x)-(\tilde{g})_{B_1}|\leq [\tilde{g}]_{C^{0,\beta}(B_1)}.$$
Thus, from \eqref{EC3.3A} we obtain
$$ \sup_{x\neq y: x, y\in B_1}\frac{|\tilde v(x)-\tilde v(y)|}{|x-y|^\beta}\leq 2 C_3 [\tilde{g}]_{C^{0,\beta}(B_1)}.$$
Reverting to the $B_r$ ball we get
\begin{equation*}
 \sup_{x\neq y: x, y\in B_r}\frac{| v_r(x)- v_r(y)|}{|x-y|^\beta}\leq 2 C_3  \sup_{x\neq y: x, y\in B_r}\frac{| g(x)- g(y)|}{|x-y|^\beta}.
\end{equation*}
This completes the proof.
\end{proof}

\section{Local $C^{1,\alpha}$ regularity}\label{S-alpha}

In this section we prove local $C^{1, \alpha}$ regularity of the minimizer $u$ to $\mathcal E(\cdot, \Omega)$. We make use of  the weak formulation in \eqref{weaksol}. The broad approach 
of this section is based on
\cite{DFM}. To make the presentation coherent,  we borrow a few notations from \cite{DFM}. Suppose that $\Omega_0\Subset\Omega_1\Subset\Omega$.
We use $w$ as a generic function and $B_\varrho(x_0)\Subset\Omega_1\Subset\Omega$. Also, let
$\mathsf{d} = \min\{\dist(\Omega, \partial \Omega_1), \dist(\Omega_1, \partial\Omega), 1\}$.

Recalling the definition of the H\"older semi-norm and the average of a function in~\eqref{def-mean}, for a function $w: A \subset \R^N \to \R^M$ for $M\geq 1$, we write
\begin{align*}
\mathrm{osc}_A w = \max_{i=1,...,M} \mathrm{osc}_A w_i \quad \mbox{and} \quad (w)_A = \Big{(} \fint_A w_i(z)dz \Big{)}_i.
\end{align*}

For $\delta \geq sq, \varrho > 0$, we denote
\begin{align*}
\snail_\delta(\varrho)&=\snail_\delta(w, B_\varrho(x_0))=\left(\varrho^\delta 
\int_{B^c_\varrho(x_0)}\frac{|w(y)-(w)_{B_\varrho(x_0)}|^q}{|y-x_0|^{n+qs}}\dy\right)^{\nicefrac{1}{q}},\; %\delta\geq sq,
\\
\av_t(w, B_\varrho(x_0))&= \left(\fint_{B_\varrho(x_0)}|w(y)-(w)_{B_\varrho(x_0)}|^t \dy\right)^{\nicefrac{1}{t}},
\\
{\rm ccp}_*(\varrho)&= \varrho^{-p}[\av_p(w, B_\varrho(x_0))]^p + \varrho^{-sq}[\av_q(w, B_\varrho(x_0))]^q
\\
&\qquad +\varrho^{-\delta}[\snail_\delta(w, B_\varrho(x_0))]^q + \norm{f}^{\frac{p}{p-1}}_{L^n(B_\varrho(x_0))} +1.
\end{align*}

We notice that by definition of $\snail_\delta$, ${\rm ccp}_*(\varrho)$ does not depend on $\delta > 0$.  We keep this notation to invoke directly the results from~\cite{DFM} that are used here.

We recall \cite[Lemma~4.1]{DFM} which also goes through in our case.
\begin{lem}\label{L4.1}
For some constant $c\equiv c(n, p, q, s)$,  we have
$$\fint_{B_{\varrho/2}(x_0)} |\grad u(x)|^p \dx +  \int_{B_{\varrho/2}(x_0)}\fint_{B_{\varrho/2}(x_0)} \frac{|u(x)-u(y)|^q}{|x-y|^{n+qs}}\dx\dy
\leq c\, {\rm ccp}_*(u, B_\varrho(x_0)),$$
for all $\varrho\in (0,  1)$ and $B_\varrho(x_0)\Subset\Omega$.
\end{lem}

The next lemma is the same as in \cite[Lemma~6.1]{DFM}
\begin{lem}\label{L4.2}
Let $x_0\in\Omega_0$.
For the solution $u\in C(\Omega)$ satisfying \eqref{weaksol} we have the following.
\begin{itemize}
\item[(i)] If $s<\beta< 1$, then
$$ \int_{B_{\varrho/2}(x_0)} \fint_{B_{\varrho/2}(x_0)} \frac{|u(x)-u(y)|^q}{|x-y|^{n+qs}}\dx\dy\leq \kappa \varrho^{(\beta-s)q},$$
for some constant $\kappa=\kappa(\data, \beta, \mathsf{d})$.

\item[(ii)] For every $t\in (0, \varrho)$ we have
$$t^{-\delta}[\snail_\delta(t)]^q=t^{-\delta}[\snail_\delta(u, B_t(x_0))]^q\leq \kappa,$$
where $\kappa=\kappa(\data, \mathsf{d})$.

\item[(iii)] If $\lambda>0$, then
$$\fint_{B_{\varrho/2}(x_0)} |\grad u|^p \dx \leq \kappa \varrho^{-\lambda p},$$
where $\kappa\equiv \kappa(\data, \mathsf{d}, \lambda)$.
\end{itemize}
\end{lem}

\begin{proof}
(i) follows from Theorem~\ref{T-main1}. Proof of (ii) is the same as in \cite[Lemma~6.1]{DFM} whereas (iii)
follows from the argument of \cite[Lemma~6.1]{DFM} using  Theorem ~\ref{T-main1} and Lemma~\ref{L4.1}.
\end{proof}

Now we are ready to prove our key lemma for the regularity estimate.
\begin{lem}\label{L4.3}
Let $h\in u+W^{1,p}_0(B_{\frac{\varrho}{4}}(x_0))$ be the solution to
$$w\mapsto \min_{w\in u+W^{1,p}_0(B_{\frac{\varrho}{4}}(x_0))}\int_{B_{\frac{\varrho}{4}}(x_0)} |\grad w|^p \dx.$$
Then for some $\sigma_2\equiv \sigma_2(n, p, q, s)\in (0, 1)$ and a constant $\kappa=\kappa(\data, \mathsf{d})$ we have
\begin{equation}\label{EL4.3A0}
\fint_{B_{\frac{\varrho}{4}}(x_0)} |\grad u-\grad h|^p\dx \leq \kappa \varrho^{\sigma_2 p},
\end{equation}
for all $\varrho\in (0, \mathsf{d}/4)$.
\end{lem}

\begin{proof}
We let $V(z)=|z|^{\frac{p-2}{2}}z$ and $\sV^2=|V(\grad u)-V(\grad h)|^2$. Also,  noticing that $h=u$ in $B^c_{\varrho/4}(x_0)$ and
setting $w=u-h$, we clearly have $w\in W^{1, p}_0(B_{\varrho}(x_0))\cap L^\infty(B_{\varrho}(x_0))$. Since $h\in C^{0, \beta}(\bar{B}_{\varrho/4}(x_0))$
by Corollary~\ref{Cor3.3}, we have $ h\in C^{0, \beta}(\bar{B}_{\varrho}(x_0))$ and 
$$[h]_{C^{0, \beta}(\bar{B}_{\varrho}(x_0))}\leq \kappa_1 [u]_{C^{0, \beta}(\bar{B}_{\varrho}(x_0))},$$
where the constant $\kappa_1=\kappa_1(n, p,\beta)$. Thus, applying Theorem~\ref{T-main1} , we can find 
a constant $\kappa_2=\kappa_2(\data, \beta, \mathsf{d})$, that satisfies
\begin{equation}\label{EL4.3A}
[h]_{C^{0, \beta}(\bar{B}_{\varrho}(x_0))}\leq \kappa_2, \quad \text{and}\quad [u]_{C^{0, \beta}(\bar{B}_{\varrho}(x_0))}\leq \kappa_2,
\end{equation}
for all $\varrho\in (0, \mathsf{d}/4)$. We choose $\beta\in (s, 1)$. It is easily seen now that
$w\in W^{s, q}(B_{\varrho}(x_0))$, and since $w=0$ in $B^c_{\varrho/4}(x_0)$, we obtain
$w\in W^{s, q}(\Rn)$ \cite[Lemma~5.1]{DNPV}. Since $w=0$ in $B^c_{\varrho/4}(x_0)$, we also have $w\in \mathbb{X}_0(\Omega)$.

Using \eqref{weaksol} and the  $p$-harmonicity of $h$, we compute
\begin{align}\label{EL4.3B}
\fint_{B_{\varrho/4}(x_0)} \sV^2 \dx &\leq \kappa_3 \fint_{B_{\varrho/4}(x_0)} (|\grad u|^{p-2}\grad u- |\grad h|^{p-2}\grad h)\cdot \grad w \dx\nonumber
\\
&= \kappa_3 \fint_{B_{\varrho/4}(x_0)} |\grad u|^{p-2}\grad u \cdot \grad w \dx\nonumber
\\
& = \kappa_ 3\fint_{B_{\varrho/4}(x_0)} f w \dx - \frac{\kappa_3}{2|B_{\varrho/4}(x_0)|}\int_{\Rn}\int_{\Rn} J_q(u(y)-u(x)) (w(y)-w(x))\frac{\dx\dy}{|x-y|^{n+sq}}\nonumber
\\
&= \kappa_3 \fint_{B_{\varrho/4}(x_0)} f w \dx - \kappa_4\int_{B_{\varrho/2}(x_0)}\fint_{B_{\varrho/2}(x_0)} J_q(u(y)-u(x)) (w(y)-w(x))\frac{\dx\dy}{|x-y|^{n+sq}}\nonumber
\\
&\qquad - 2\kappa_4 \int_{B^c_{\varrho/2}(x_0)}\fint_{B_{\varrho/2}(x_0)} J_q(u(y)-u(x)) (w(y)-w(x))\frac{\dx\dy}{|x-y|^{n+sq}}\nonumber
\\
&:= A_1 + A_2 + A_3,
\end{align}
where $\kappa_3=\kappa_3(n, p)$.
From the Sobolev inequality
\begin{align*}
|A_1|\leq \kappa_3 \norm{f}_\infty \fint_{B_{\varrho/4}(x_0)} |w| \dx &\leq \norm{f}_\infty \fint_{B_{1}(0)} |w(x_0+\frac{\varrho}{4} x)| \dx
\\
&\leq \kappa_3 \norm{f}_\infty \left(\fint_{B_{1}(0)} |w(x_0+ \frac{\varrho}{4} x)|^p \dx\right)^{\nicefrac{1}{p}}
\\
&\leq \kappa_3 \varrho \norm{f}_\infty \left(\fint_{B_{1}(0)} |\grad w(x_0+ \frac{\varrho}{4} x)|^p \dx\right)^{\nicefrac{1}{p}}
\\
&= \kappa_3 \varrho \norm{f}_\infty \left(\fint_{B_{\varrho/4}(x_0)} |\grad w(x)|^p \dx\right)^{\nicefrac{1}{p}}
\\
&= \kappa_5 \varrho^{1-n}  \norm{f}_\infty \left(\norm{\grad u}_{L^p(B_{\varrho/4}(x_0))} + \norm{\grad h}_{L^p(B_{\varrho/4}(x_0))}\right)
\\
&\leq  2\kappa_5 \varrho^{1-n}  \norm{f}_\infty \norm{\grad u}_{L^p(B_{\varrho/4}(x_0))} 
\\
&\leq \kappa_6 \norm{f}_\infty \varrho^{1-\lambda},
\end{align*}
for some constant $\kappa_6\equiv\kappa_6(\data, \mathsf{d}, \lambda)$, where in the sixth line we use minimizing property of $h$ and
in the last estimate we use Lemma~\ref{L4.2}(iii).

Using \eqref{EL4.3A} we see that
$$\int_{B_{\varrho/2}(x_0)} \fint_{B_{\varrho/2}(x_0)} \frac{|w(x)-w(y)|^q}{|x-y|^{n+qs}}\dx\dy\leq \kappa_1 \varrho^{q(\beta-s)}$$
for some $\kappa_1\equiv\kappa_1(\data, \beta, \mathsf{d})$. Therefore,  from Lemma~\ref{L4.2}(i), we obtain 
\begin{align*}
|A_2| & \leq \kappa_2 \varrho^{(\beta-s)(q-1)} \left(\int_{B_{\varrho/2}(x_0)} \fint_{B_{\varrho/2}(x_0)} \frac{|w(x)-w(y)|^q}{|x-y|^{n+sq}}\dx\dy\right)^{\nicefrac{1}{q}}
\\
&\leq \kappa_3  \varrho^{(\beta-s)(q-1)} \varrho^{\beta-s} =\kappa_3 \varrho^{q(\beta-s)},
\end{align*}
where $\kappa_3\equiv\kappa_3(\data, \beta, \mathsf{d})$.

Since $h$ is the minimizer in the ball $B_{\varrho/4}(x_0)$ with boundary data $u$, we get
$$\norm{w}_{L^\infty(B_{\varrho/4}(x_0))}\leq \sup_{B_{\varrho/4}(x_0)} |h(x)- u(x)|\leq {\rm osc}_{B_{\varrho/4}(x_0)} (h)
+ {\rm osc}_{B_{\varrho/4}(x_0)} (h) \leq \kappa_4 \varrho^\beta,$$
where $\kappa_4=\kappa_4(\data, \beta, \mathsf{d})$.
Again, using $w=0$ in $B^c_{\varrho/4}(x_0)$, we see that
\begin{align*}
&\left|\int_{B^c_{\varrho/2}(x_0)}\fint_{B_{\varrho/2}(x_0)} J_q(u(y)-u(x)) (w(y)-w(x))\frac{\dx\dy}{|x-y|^{n+sq}}\right|
\\
& \leq  \int_{B^c_{\varrho/2}(x_0)}\fint_{B_{\varrho/4}(x_0)} |u(y)-u(x)|^{q-1} |w(x)|\frac{\dx\dy}{|x-y|^{n+sq}}
\\
&\leq \kappa_1 \varrho^{\beta} \Bigl[ \int_{B^c_{\varrho/2}(x_0)}\fint_{B_{\varrho/4}(x_0)} |u(y)-(u)_{B_{\varrho}(x_0)}|^{q-1}\frac{\dx\dy}{|x-y|^{n+sq}} +
\\
&\quad  \int_{B^c_{\varrho/2}(x_0)}\fint_{B_{\varrho/4}(x_0)} |u(x)-(u)_{B_{\varrho}(x_0)}|^{q-1}\frac{\dx\dy}{|x-y|^{n+sq}} \Bigr]
\\
&:= \kappa_1 \varrho^{\beta} (H_1+H_2).
\end{align*}
Using Lemma~\ref{L4.2}(ii), we see that
$$ H_1\leq \kappa_2 \varrho^{-s} (t^{-\delta}[\snail_\delta(t)]^q)^{\frac{q-1}{q}}\leq \kappa_3 \varrho^{-s},$$
and using \eqref{EL4.3A}
\begin{align*}
H_2\leq \kappa_2 \varrho^{\beta(q-1)}  \int_{B^c_{\varrho/2}(x_0)}\fint_{B_{\varrho/4}(x_0)} \frac{\dx\dy}{|x-y|^{n+sq}}
\leq \kappa_3 \varrho^{\beta(q-1)-sq}.
\end{align*}
Thus, gathering the terms, we obtain $|A_3|\leq \kappa_5 \varrho^{\beta-s}$ for some $\kappa_5\equiv\kappa_5(\data, \beta,\mathsf{d})$.

Set $\sigma_1=\frac{1}{p}\min\{(\beta-s), 1-\lambda\}$ and $\lambda\in (0, 1)$. From \eqref{EL4.3B}, this leads to
\begin{equation}\label{EL4.3C}
\fint_{B_{\varrho/4}(x_0)} \sV^2 \dx\leq \kappa \varrho^{p\sigma_1}
\end{equation}
for all $\varrho\in (0, \mathsf{d}/4)$. Again, since $|a-b|^p\leq |V(a)-V(b)|^2$ for $p\geq 2$ (see \cite[Lemma~A.3]{BLS18}), \eqref{EL4.3A0} follows by taking $\sigma_1=\sigma_2$. For $p\in (1, 2)$, we use the inequality (see \cite[p.~74]{Lindq})
$$|V(a)-V(b)|\geq \frac{p}{2} (1+|a|^2+|b|^2)^{\frac{p-2}{4}} |a-b|,$$
 to estimate
\begin{align*}
\fint_{B_{\frac{\varrho}{4}}(x_0)} |\grad u-\grad h|^p\dx &\leq \kappa_p \fint_{B_{\frac{\varrho}{4}}(x_0)} \sV^{p/2} (1+|\grad u|+|\grad h|)^{\frac{p(2-p)}{2}}\dx
\\
&\leq \kappa_p \left( \fint_{B_{\frac{\varrho}{4}}(x_0)} \sV^2 \dx\right)^{\frac{p}{2}}\cdot 
\left(\fint_{B_{\frac{\varrho}{4}}(x_0)} (1+|\grad u|+|\grad h|)^p\dx\right)^{\frac{2-p}{2}}
\\
& \leq \kappa_1 \left(1+ \fint_{B_{\frac{\varrho}{4}}(x_0)} |\grad u|^p\dx\right)^{\frac{2-p}{2}} \varrho^{\sigma_1 \frac{p^2}{2}}
\\
&\leq \kappa_2 \varrho^{\sigma_1 \frac{p^2}{2}-\lambda\frac{(2-p)p}{2}},
\end{align*}
for some constant $\kappa_2\equiv\kappa_2(\data, \lambda, \mathsf{d})$, where in the third line we use \eqref{EL4.3C} and minimizing property of $h$, and in the
last line we use Lemma~\ref{L4.2}(iii). Now, we can choose $\lambda$ small enough so that 
$\sigma_2:=\sigma_1 p/2-\lambda\frac{(2-p)}{2}>0$. 
\end{proof}

%\begin{thm}
%We have $u\in C^{1, \alpha}_{\rm loc}(\Omega)$.
%\end{thm}
Now we provide the proof of $C^{1, \alpha}$ regularity in $\Omega_0$.
\begin{proof}[Proof of Theorem~\ref{T-main}]
The key ingredient is Lemma~\ref{L4.3}. Consider $h$ from Lemma~\ref{L4.3}.
We recall the following estimate of $h$ from \cite{Man86,Man88}
\begin{equation}\label{EA2.10}
{\rm osc}_{B_t(x_0)} (\grad{h}) \leq \kappa \left(\frac{t}{\varrho}\right)^{\alpha_0} \left( \fint_{B_{\varrho/4}(x_0)} |\grad h|^p \right)^{\nicefrac{1}{p}}
\end{equation}
for $0<t\leq \varrho/8$ and for some $\alpha_0\equiv\alpha_0(n ,p)\in (0, 1)$, $\kappa\equiv\kappa(n, p)$. Now we compute, using \eqref{EA2.10}
and minimality of $h$, that
\begin{align*}
\fint_{B_t(x_0)} |\grad u- (\grad u)_{B_t(x_0)}|^p \dx &\leq 2^p \left({\rm osc}_{B_t(x_0)} \grad{h}\right)^p + 2^p \fint_{B_t(x_0)} |\grad u-\grad h|^p\dx
\\
&\leq  2^p \left({\rm osc}_{B_t(x_0)} \grad{h}\right)^p  + 2^{p - 2n} \left[\frac{\varrho}{t}\right]^n \fint_{B_{\frac{\varrho}{4}}(x_0)} |\grad u-\grad h|^p
\\
& \leq \kappa \left(\frac{t}{\varrho}\right)^{p \alpha_0} \left( \fint_{B_{\varrho/4}(x_0)} |\grad u|^p + \kappa \varrho^{\sigma_2 p} \right) + \kappa \left[\frac{\varrho}{t}\right]^n \varrho^{\sigma_2 p}
\\
&\leq \kappa \left(\frac{t}{\varrho}\right)^{p \alpha_0} \varrho^{-\lambda p} + \kappa \left[\frac{\varrho}{t}\right]^n \varrho^{\sigma_2 p},
\end{align*}
where the last inequality follows from Lemma~\ref{L4.2}(iii). Now set 
$$\lambda=\frac{\sigma_2 p\alpha_0}{4n}\quad \text{and} \quad t=\frac{1}{8}\varrho^{1+\frac{p\sigma_2}{2n}},$$
to obtain
$$ \fint_{B_t(x_0)} |\grad u- (\grad u)_{B_t(x_0)}|^p \dx \leq \kappa t^{\alpha p}, \quad \text{where}\quad \alpha=\frac{p\sigma_2\alpha_0}{4n+2\sigma_2p},$$
where $\kappa\equiv\kappa(\data, \mathsf{d})$.
This is the standard Campanato criterion which gives $C^{0,\alpha}$ regularity of $\grad{u}$, proving $C^{1, \alpha}$ regularity of $u$ in $\Omega_0$.
\end{proof}

\subsection*{Acknowledgement}
Part of this project was done during a visit of A.B.\ at the Instituto de Matem\'{a}tica of Universidade Federal do Rio de Janeiro. The kind hospitality of the department is acknowledged.
This research of Anup Biswas was supported
in part by a SwarnaJayanti fellowship SB/SJF/2020-21/03. Erwin Topp was supported by CNPq
Grant 306022 and FAPERJ APQ1 Grant 210.573/2024. Both authors were also supported by a CNPq Grant 408169.

\subsection*{Conflict of interest} The authors declare to have no conflict of interests. No data are attached to this paper.

%%%%%%%%%%%%%%%%%%%%%%%%%%%%%%%%%%%%%%%%%%%%%%%%%%%%%%%%%%%%%%%%%%%%%%%%%%
%\bibliographystyle{plain}
%\bibliography{ref.bib}

\end{document}